\documentclass{amsart}
\usepackage{amsmath, amssymb, enumerate} 

\usepackage{colonequals}

\usepackage{eucal}

\usepackage{hyperref}

\makeatletter
\@namedef{subjclassname@2020}{\textup{2020} Mathematics Subject Classification}
\makeatother

\usepackage{tikz-cd}
\usetikzlibrary{decorations.pathmorphing}
\usepackage[nobysame]{amsrefs}

\numberwithin{equation}{section}

\theoremstyle{plain}
\newtheorem{theorem}[equation]{Theorem}
\newtheorem{ithm}{Theorem}

\newtheorem{proposition}[equation]{Proposition}
\newtheorem{lemma}[equation]{Lemma}
\newtheorem{corollary}[equation]{Corollary}

\theoremstyle{definition}

\newtheorem{example}[equation]{Example}
\newtheorem{conjecture}[equation]{Conjecture}
\newtheorem{icon}{Conjecture}

\newtheorem{chunk}[equation]{}

\theoremstyle{remark}

\newtheorem*{ack}{Acknowledgements}

\newtheorem*{claim}{Claim}

\newcommand*{\intref}[2]{\def\tmp{#1}\ifx\tmp\empty\hyperref[#2]{\ref*{#2}}\else\hyperref[#2]{#1~\ref*{#2}}\fi}

\newcommand{\ann}{\operatorname{ann}}

\newcommand{\CH}[3]{\operatorname{H}^{#1}(#2, #3)}

\newcommand{\chr}{\operatorname{char}}
\newcommand{\cat}[1]{\mathcal{#1}}

\newcommand{\depth}{\operatorname{depth}}
\newcommand{\euler}[1]{{\mathbb E}(#1)}
\newcommand{\eulerk}[1]{{\mathbb K}(#1)}

\newcommand{\gr}{\operatorname{gr}}
\newcommand{\grade}{\operatorname{grade}}
\newcommand{\grog}[1]{\operatorname{G}_{0}(#1)}
\newcommand{\rgrog}[1]{\overline{\operatorname{G}_{0}}(#1)}

\newcommand{\kos}[2]{K(#1;#2)}
\newcommand{\hh}[1]{\operatorname{H}(#1)}
\newcommand{\HH}[2]{\operatorname{H}_{#1}(#2)}

\newcommand{\height}{\operatorname{height}}
\newcommand{\hilb}{\operatorname{Hilb}}
\renewcommand{\hom}{\operatorname{Hom}}
\newcommand{\knot}[1]{\operatorname{K}_0^{\mathfrak{m}}(#1)}
\newcommand{\rknot}[1]{\overline{\operatorname{K}_0^{\mathfrak{m}}}(#1)}
\newcommand{\length}{\ell}
\newcommand{\lch}[3]{\operatorname{H}_{#2}^{#1}(#3)}

\newcommand{\pair}[2]{\langle #1,#2\rangle}
\newcommand{\pdim}{\operatorname{proj\,dim}}
\newcommand{\perf}[1]{\mathcal{F}^{\mathfrak{m}}(#1)}
\newcommand{\pos}[1]{[\!|#1]\!]}
\newcommand{\proj}[1]{\operatorname{Proj}(#1)}

\newcommand{\rank}{\operatorname{rank}}

\newcommand{\shift}{\mathsf{\Sigma}}
\newcommand{\spec}{\operatorname{Spec}}

\newcommand{\syz}{\Omega}

\newcommand{\tor}{\operatorname{Tor}}

\newcommand{\cale}{{\mathcal E}}
\newcommand{\calf}{{\mathcal F}}
\newcommand{\call}{{\mathcal L}}
\newcommand{\calm}{{\mathcal M}}
\newcommand{\caln}{{\mathcal N}}

\newcommand{\calo}{{\mathcal O}}
\newcommand{\calu}{{\mathcal U}}

\newcommand{\bbP}{{\mathbb P}}
\newcommand{\bbQ}{{\mathbb Q}}
\newcommand{\bbZ}{{\mathbb Z}}
\newcommand{\bbR}{{\mathbb R}}

\newcommand{\fm}{\mathfrak{m}}
\newcommand{\fn}{\mathfrak{n}}

\newcommand{\fp}{\mathfrak{p}}

\newcommand{\bsa}{\boldsymbol a}
\newcommand{\bsb}{\boldsymbol b}
\newcommand{\bsd}{\boldsymbol d}
\newcommand{\bsr}{\boldsymbol r}
\newcommand{\bss}{\boldsymbol s}
\newcommand{\bsx}{\boldsymbol x}
\newcommand{\bsy}{\boldsymbol y}

\newcommand{\vf}{\varphi}

\newcommand{\lra}{\longrightarrow}
\newcommand{\xra}{\xrightarrow}

\begin{document}

\title[Multiplicities in local algebra]{Multiplicities and Betti numbers \\ in local algebra via lim Ulrich points}

\author[Iyengar]{Srikanth B. Iyengar}
\address{Department of Mathematics, University of Utah, Salt Lake City, UT 84112, USA}
\email{iyengar@math.utah.edu}

\author[Ma]{Linquan Ma}
\address{Department of Mathematics, Purdue University, 150 N. University street, IN 47907}
\email{ma326@purdue.edu}

\author[Walker]{Mark E. Walker}
\address{Department of Mathematics, University of Nebraska, Lincoln, NE 68588, U.S.A.}
\email{mark.walker@unl.edu}

\begin{abstract}
This work concerns  finite free complexes with finite length homology over a  commutative noetherian local ring $R$. The focus is on  complexes that have length $\mathrm{dim}\, R$, which is the smallest possible value, and in particular on free resolutions of modules of finite length and finite projective dimension.   Lower bounds are obtained on the Euler characteristic of such short complexes when $R$ is a strict complete intersection, and also on the Dutta multiplicity, when $R$ is the localization  at its maximal ideal of a standard graded algebra over a field of positive prime characteristic. The key idea in the proof is the construction of a suitable Ulrich module, or, in the latter case, a sequence of modules that have the Ulrich property asymptotically, and with good convergence properties in the rational Grothendieck group of $R$. Such a sequence is obtained by  constructing an appropriate sequence of sheaves on the associated projective variety.
\end{abstract}

\date{\today}

\keywords{complete intersection ring,  Dutta multiplicity, Euler characteristic, finite free complex, finite projective dimension, lim Ulrich sequence}
\subjclass[2020]{13D40 (primary); 13A35, 13C14, 13D15,  14F06  (secondary)}

\maketitle

\setcounter{tocdepth}{1}
\tableofcontents

\section{Introduction}
This paper investigates various questions about multiplicities of modules over (commutative noetherian) local rings, the most basic of which is:
\begin{quote}
  For a given local ring $R$, what is the smallest possible  length of a nonzero $R$-module having finite projective dimension?
\end{quote}

This question is only interesting for Cohen-Macaulay rings: If $R$ is not Cohen-Macaulay, then every nonzero module of finite projective dimension has infinite length; this is by Roberts' New Intersection Theorem~\cite{Roberts:1987}. On the other hand, when $R$ is Cohen-Macaulay  and $\bsx\colonequals x_1, \dots, x_d $ form a system of parameters for $R$, the $R$-module $R/(\bsx)$ has finite projective dimension and its length  satisfies
\[
\length_R(R/(\bsx)) \geq e(R)
\]
where $e(R)$ denotes the Hilbert-Samuel multiplicity of $R$. In fact, at least when the residue field of $R$ is infinite, equality holds when $\bsx$ is chosen to be sufficiently general.   This leads us to  conjecture:

\begin{icon}
\label{intro:conjA}
For a local ring $R$, every nonzero $R$-module $M$ of finite projective dimension satisfies $\length_R M \geq e(R)$.
\end{icon}

One measure of its difficulty is that it  implies the still open Lech's conjecture~\cite{Lech:1960}, recalled in Section~\ref{se:lech}, at least  for Cohen-Macaulay rings;  see \cite[Chapter V]{Ma:2014}.

It is not hard to verify Conjecture~\ref{intro:conjA} when $\dim(R) \leq 1$; see Example \ref{ex:dim1}.  It also holds trivially when $R$ is regular, for then $e(R) = 1$. For an arbitrary local ring the conjecture holds when the module $M$ is extended from a module over a regular ring: if $A \to R$ is a finite flat map of local rings with $A$ regular, $N$ is a finitely generated $A$-module, and $M = R \otimes_A N$, then $\length_R M \geq e(R)$.

Nearly all cases in which we can establish Conjecture \ref{intro:conjA} for a local ring $R$, it is because $R$ admits an Ulrich module: a nonzero maximal Cohen-Macaulay  $R$-module $U$ whose Hilbert-Samuel multiplicity, $e(U)$, equals its minimal number of generators, $\nu_R(U)$; this is the smallest possible value for the multiplicity. For instance, in Theorem~\ref{th:UlrichspanR} we prove the following result. Here  $\euler R$ denotes the Grothendieck group of finitely generated $R$-modules modulo numerical equivalence, extended to $\bbR$; see Section~\ref{se:grog} for details. Moreover  $\beta^R_i(M)$ is the $i$th Betti number of $M$.

\begin{ithm}
\label{intro:thmA}
If $R$  admits an Ulrich module $U$ whose class $[U]$ in $\euler R$ is a multiple of $[R]$, then for each $R$-module $M$ of finite length and finite projective dimension 
 \[
  \binom{\dim R} i  \length_R M \geq \beta^R_i(M) e(R)
 \]
for all $i$. In particular,  Conjecture \ref{intro:conjA} holds for $R$.
\end{ithm}

A local ring $R$ is a \emph{strict complete intersection} if its associated graded ring $\mathrm{gr}_\fm(R)$ is a complete intersection. Every strict complete intersection is a complete intersection;  the converse holds for hypersurfaces, but fails if the codimension is at least two.  It follows from a result of Backelin, Herzog, and Ulrich \cite{Backelin/Herzog/Ulrich:1991} that the hypothesis of Theorem \ref{intro:thmA} holds whenever $R$ is a strict complete intersection. Thus Conjecture~\ref{intro:conjA} holds for all such rings and, in particular, for all hypersurfaces.

Previously, Avramov, Buchwietz, Iyengar and Miller \cite[Section 1]{Avramov/Buchweitz/Iyengar/Miller:2010a} established lower bounds on the Loewy length of modules of finite projective dimension.  One of their results is that if $R$ is a strict complete intersection of codimension $c$, then
\[
\ell\ell_R M \geq e(R) -c + 1\,.
\]
In particular, $\ell_R M \geq e(R) - c + 1$, which also verifies Conjecture \ref{intro:conjA} for hypersufaces.

It is an open question whether  every Cohen-Macaulay ring admits a nonzero Ulrich module. This appears to be a rather difficult problem; in any case, known examples are rare. However,  for many arguments involving Ulrich modules what is needed is not an actual Ulrich module, but rather just a class in the real vector space $\euler R$  with suitable properties.  In detail, we say that a nonzero class $\alpha$ in $\euler R$ is  a  \emph{lim Ulrich point} if it is a positive linear combination of terms of the form
\[
\lim_{n \to \infty} \frac{[U_n]}{\nu_R(U_n)}
\]
where $(U_n)_{n\geqslant 0}$ is a \emph{lim Ulrich sequence} of $R$-modules; that the limit above exists is an additional constraint. We refer the reader to \ref{ch:lus}  for the precise definition of lim Ulrich sequences, but, heuristically, it means that the  modules in the list have the Ulrich property asymptotically.

For instance, the conclusion of Theorem \ref{intro:thmA} remains valid if, instead of requiring that $[R] = \frac{1}{r} [U]$ in $\euler R$ for some integer $r$ and Ulrich module $U$, we have merely that $[R]$ is a lim Ulrich point in $\euler R$.

Lim Ulrich sequences appear to be more common than Ulrich modules. In particular, the second author has shown that every standard graded algebra over a perfect field of positive characteristic admits a lim Ulrich sequence of modules. We revisit the construction of such sequence in Sections \ref{se:lim-Ulrich-sheaves} and \ref{se:Ma}, approaching them from a more geometric, and hence simpler, point of view.

In particular, we introduce the notion of a ``lim Ulrich sequence of sheaves'' on a projective scheme over a field. We prove a lim Ulrich sequence of sheaves  gives rise  to a lim Ulrich sequence of graded modules over the associated homogeneous coordinate ring; see Theorem \ref{th:usheaf-umodule}. Moreover, in Theorem \ref{th:ma-revisited} we establish that every projective scheme over an infinite, perfect field of positive characteristic $p$ admits a lim Ulrich sequence of sheaves. These two results recover the second author's construction of lim Ulrich sequence of modules over such rings.   We hope that our results on lim Ulrich sequence of sheaves  prove to be of wider interest; for instance, in a forthcoming paper we intend to use them to study the cone of cohomology tables of coherent sheaves on projective schemes.

To summarize, if $R$ is the localization of a standard graded $k$-algebra $A$ at its homogeneous maximal ideal, where $k$ is a perfect field of positive characteristic, then $R$ admits a lim Ulrich sequence $(U_n)_{n\geqslant 0}$. Moreover, we are able to analyze the associated sequence of points in the Grothendieck group $\euler R$. Alas, this sequence fails to establish that $[R]$ is a lim Ulrich point. It is therefore  not possible to deduce Conjecture \ref{intro:conjA} using this sequence of modules.

Rather, we prove that $[R]_d \in \euler R$ is a lim Ulrich point,  where $[R]_d$ is the $d$-th component of $[R]$ with respect to the weight decomposition of $\euler R$ induced by the action of  Frobenius. This leads to the result below where $\chi_\infty(M)$ denotes the Dutta multiplicity of $M$.

\begin{ithm}
\label{intro:thmB}
If $R$ is the localization at its homogeneous maximal ideal of a standard graded algebra over a field of positive characteristic,  and $M$ is a finitely generated $R$-module $M$  of finite length and finite projective dimension, then
 \[
  \binom{\dim R} i  \chi_\infty(M) \geq \beta^R_i(M)  e(R)
\]
for all $i$.  In particular, $\chi_\infty(M) \geq e(R)$.
\end{ithm}

One has $\chi_\infty(M)=\length_R(M)$ when $R$ is a complete intersection, and more generally when $R$ is  numerically Roberts, and also when $M$ is the localization of a finitely generated graded module. This connects the theorem above to Theorem~\ref{intro:thmA}. 

 Theorem~\ref{intro:thmB} is in turn subsumed in Theorem~\ref{intro:thmC} below that applies also to non Cohen-Macaulay rings. As mentioned earlier, when $R$ is not Cohen-Macaulay every nonzero module of finite projective dimension has infinite length, and so our investigations for such rings focus instead on ``short complexes supported on the maximal ideal''. This term refers to a non-exact complex of finite  free $R$-modules of the form
\[
F \colonequals  0 \lra F_{\dim(R)} \lra \cdots \lra F_1 \lra F_0 \lra 0\,,
\]
such that $H_i(F)$ has finite length for all $i$.  The adjective  ``short'' comes from the fact that the length of any finite free complex with nonzero finite length homology is at least $\dim R$; this is the New Intersection Theorem. When $R$ is Cohen-Macaulay, such a short complex is necessarily the resolution of the module $H_0(F)$, and thus we are back in our original context.

For any finite free complex $F$ with finite length homology,  let $\chi(F)$ denote its Euler characteristic.  When $R$ has positive characteristic $p$ and $\varphi\colon R\to R$ is the Frobenius endomorphism, the \emph{Dutta multiplicity} of $F$ is
\[
\chi_\infty(F) \colonequals \lim_{n \to \infty}  \frac{\chi((\vf^*)^n(F))}{p^{n \dim(R)}}\,.
\]
Theorem \ref{intro:thmB} extends to non Cohen-Macaulay rings as follows; see Corollary \ref{co:chi-infinity}.

\begin{ithm}
\label{intro:thmC}
If $R$ is the localization at its homogeneous maximal ideal of a standard graded algebra over a field of positive characteristic, and $F$ is a short complex supported on the maximal ideal, then
\[
  \binom{\dim R} i  \chi_\infty(F) \geq \beta_i(F)  e(R)
 \]
 for all $i$.  In particular, $\chi_\infty(F) \geq e(R)$.
\end{ithm}

Emboldened by  this theorem, we pose:

\begin{icon}
\label{intro:conjB}
 For a local ring $R$, every short complex $F$ supported on the maximal ideal satisfies
\[
  \chi_\infty(F) \geq e(R)\,.
 \]
  \end{icon}

We close this introduction with a few remarks. First, Conjecture \ref{intro:conjB} implies Lech's conjecture; see Proposition~\ref{pr:lech}.   Second, Roberts \cite{Roberts:1989} has constructed short complexes $F$ supported on the maximal ideal of a certain local ring such that $\chi(F) < 0$. Therefore Conjecture \ref{intro:conjB} cannot hold with the usual Euler characteristic $\chi$ in place of the Dutta multiplicity $\chi_\infty$.

Finally, Yhee \cite{Yhee:2021} has recently constructed two-dimensional, complete, local domains $R$ that admit no lim Ulrich sequences. The existence of such rings points to a limitation of the central technique used in this paper for addressing Conjectures \ref{intro:conjA} and \ref{intro:conjB}.
Nevertheless, both conjectures hold for Yhee's examples. These topics are also discussed in Section \ref{se:lech}.

\begin{ack}
Our thanks to a referee for comments and corrections on an earlier version of this manuscript, and to Yhee for sharing a preliminary version of \cite{Yhee:2021}. The authors were partly supported by National Science Foundation grants DMS-2001368 (SB), DMS-1901672 and FRG DMS-1952366 (LM), and DMS-1901848 (MW). LM was also partly supported by a fellowship from the Sloan Foundation.
\end{ack}

\section{Mulitiplicities and finite projective dimension}
In this section we recall basic definitions and results concerning multiplicities and finite free complexes. We take the book of Bruns and Herzog~\cite{Bruns/Herzog:1998} as our standard reference for this material. Throughout $(R, \fm, k)$ is a local ring with maximal ideal $\fm$ and residue field $k$.  Set $d\colonequals \dim R$.  Given a finitely generated $R$-module $M$, set
\[
\nu_R(M)\colonequals \rank_k(M/\fm M)\,;
\]
this is the size of a minimal generating set for $M$. The Krull dimension of $M$ is denoted $\dim_RM$.

\begin{chunk}
\label{ch:ed-symbol}
The key invariant in this work is the (modified) multiplicity $e_d(-)$, defined on the category of finitely generated $R$-modules by the formula:
\[
e_d(M) \colonequals d! \lim_{n \to \infty} \frac{\length_R(M/\fm^{n}M)}{n^{d}}\,,
\]
where $\length_R(-)$ denotes length. When there is no cause of confusion we write $e(R)$ instead of $e_d(R)$. It is a consequence of a theorem of Hilbert and Serre that $e_d(M)$ is a nonnegative integer. Moreover, $e_d(M)=0$ if and only if $\dim_R M< \dim R$.  An important property of the multiplicity  is that it is additive on short exact sequence; see \cite[Corollary~4.7.7]{Bruns/Herzog:1998} for details.

In a few places we need also the \emph{multiplicty} of $M$, defined by
\[
e(M) \colonequals (\dim_R M)! \lim_{n \to \infty} \frac{\length_R(M/\fm^{n}M)}{n^{\dim_RM}}\,.
\]
One has $e(M)=e_d(M)$ if and only if $\dim_RM=\dim R$.
\end{chunk}

\begin{chunk}
Given a sequence $\bsr\colonequals r_1,\dots,r_n$ we write $\kos{\bsr}M$ for the Koszul complex on $\bsr$ with coefficients in $M$, and
$\HH i{\bsr;M}$ for its homology module in degree $i$.  When $\bsr$ is a \emph{multiplicity system} for $M$, meaning that $\length_R(M/\bsr M)$ is finite, the lengths of all the Koszul homology modules are finite, and in this case one sets
\[
\chi(\bsr;M)\colonequals \sum_i  (-1)^i \length_R  \HH i{\bsr;M} \,.
\]
If $\bsr \colonequals r_1,\dots,r_d$ generate a minimal reduction of $\fm$, in the sense of \cite[Remark~4.6.9]{Bruns/Herzog:1998}, then
\begin{equation}
\label{eq:serre}
e_d(M) = \chi(\bsr;M)\,.
\end{equation}
This result follows from the work of Auslander and Buchsbaum, and Serre~\cite[Theorem~4.7.6 and Corollary 4.6.10]{Bruns/Herzog:1998}. Such minimal reductions exist when the residue field $k$ is infinite, so in this case one can compute the multiplicity as an Euler characteristic; see Corollary~\ref{co:ed-symbol} for a version that covers all local rings.

In the same vein when $k$ is infinite, there exist elements $\bsr$ in $\fm$ that form a system of parameters for $M$ and satisfy
\[
e(M) = \chi(\bsr; M) \,.
\]
When in addition $M$ is Cohen-Macaulay any system of parameters for $M$ is a regular sequence and hence for a sequence $\bsr$ as above one gets that
\[
e(M) = \length_R(M/\bsr M)\,.
\]
\end{chunk}

\begin{chunk}
\label{ch:codim}
Let $M$ be a finitely generated $R$-module. With $I\colonequals \ann_RM$,  the annihilator ideal of $M$, and $\grade_RM$ denoting  the longest regular sequence in  $I$, one has
\[
\depth R - \dim_R M \leq \grade_RM  \leq \height I
\leq \dim R-\dim_R M\leq \pdim_RM\,.
\]
For a proof, see \cite[(2.3) and (2.4)]{Avramov/Foxby:1998}. The rightmost inequality is useful only when $\pdim_RM$ is finite, and then it is the Intersection Theorem of Peskine and Szpiro~\cite[\S1]{Peskine/Szpiro:1973}, proved in full generality by Roberts as a consequence of his New Intersection Theorem; see \ref{ch:short}.  When $\pdim_RM$ is finite the equality of Auslander and Buchsbaum reads
\[
\pdim_RM = \depth R - \depth_RM\,.
\]
\end{chunk}
A finitely generated $R$-module $M$ is said to be  \emph{perfect} if
\[
\grade_RM=\pdim_RM\,.
\]
It is immediate from the inequalities in \ref{ch:codim} that when the ring $R$ is Cohen-Macaulay, an $R$-module $M$ is perfect if and only if it Cohen-Macaulay and of finite projective dimension. Conversely if $R$ has a nonzero module of finite length and finite projective dimension, then $R$ is Cohen-Macaulay. This is a consequence of the New Intersection Theorem, recalled below; see ~\ref{ch:short}.

\begin{chunk}
\label{ch:finite-free}
We write $\perf R$ for the category of finite free $R$-complexes, that is to say,  bounded complexes of finitely generated free $R$-modules, with finite length homology. The \emph{Euler characteristic} of such an $F$ is the integer
\[
\chi_R(F)  \colonequals  \sum_i(-1)^i \length_R\HH iF\,.
\]
For any integer $i$, the $i$th \emph{Betti number} of $F$ is
\[
\beta^R_i(F) \colonequals  \rank_k \tor^R_i(k,F)\,.
\]
When  $F$ is minimal, that is to say,  $d(F)\subseteq \fm F$, one has $\beta^R_i(F)=\rank_R F_i$. Set
\[
\beta^R(F)\colonequals \sum_i \beta^R_i(F)\,,
\]
this is the \emph{total Betti number} of $F$.
\end{chunk}

\begin{chunk}
\label{ch:short}
Consider $F$ in $\perf R$ of the form
\[
  0 \lra F_n \lra \cdots  \lra F_0 \lra 0
\]
and with $\HH 0F\ne 0$.  The New Intersection Theorem proved by Roberts~\cites{Roberts:1987, Roberts:1989} states that  $n \ge\dim R$. In what follows we say $F$ is a \emph{short} complex in $\perf R$ to indicate that it is a finite free complex as above  with $n=\dim R$.  When $R$ is Cohen-Macaulay this is tantamount to saying that $F$ is the  free resolution of a finite length $R$-module, namely, $\HH 0F$.

For any  complex $F$ in $\perf R$ as above and $R$-module $U$ one has
\[
\depth_RU =  \dim R - \sup\{i\mid \HH i{F\otimes_RU}\ne 0\}\,.
\]
This is a particular instance of the depth sensitivity of complexes in $\perf R$; see, for example, \cite[Corollary~6.4]{Iyengar:1999}.  In particular $\HH i{F\otimes_RU}=0$ for $i\ge 1$  when $U$ is maximal Cohen-Macaulay.
\end{chunk}

\begin{chunk}
\label{ch:gonflament}
Let $\vf\colon (R,\fm) \to (S,\fn)$ be a flat local map with $\fm  S=\fn$.  It is easy to check from definitions that $e(S) = e(R)$. Equally, for any $F\in \perf R$ it is easy to verify that the complex $S\otimes_RF$ is in ${\mathcal{F}^{\fn}(S)}$, and there is an equality
\[
\beta^S_i(S\otimes_RF) = \beta^R_i(F) \qquad\text{for each $i$.}
\]
Moreover $\chi(S\otimes_RF) = \chi(F)$.  Given any extension of fields $k\to l$, there exists a flat local extension $S$ of $R$ as above whose residue field is $l$; see, for example, \cite[Appendice~2]{BourbakiCAIX:2006}. This, and the discussion in the previous paragraph, often allow us to reduce the problem on hand to the case where the residue field of $R$ is infinite.
\end{chunk}

\section{Euler characteristics of short complexes}

Let $(R,\fm,k)$ be a local ring and set $d\colonequals \dim R$.  We shall be interested in those rings $R$ with the property that the Euler characteristic of any short complex  $F$ in $\perf R$, in the sense of \ref{ch:short}, satisfies inequalities:
\begin{equation}
\label{eq:euler}
\binom di \chi_R(F) \ge \beta^R_i(F)  e(R) \quad\text{for each $0\le i\le d$}\,.
\end{equation}
If $R$ is Cohen-Macaulay, $F$ has homology only in degree $0$ so $\chi_R(F)=\length_R(\HH 0F)$. Thus  Conjecture~\ref{intro:conjA}  is the  case $i=0$ of the inequalities above.

The inequalities in \eqref{eq:euler} do not always hold: Roberts~\cite[\S4]{Roberts:1989} has constructed  $F$ as above with negative Euler characteristic.  However when $R$ is Cohen-Macaulay,  the negativity of the Euler characteristic cannot arise as an obstruction, and we do not know if the inequalities above hold for this class of rings.  They are easy to verify when $R$ is regular. The main result of this section, Theorem~\ref{th:sci},  is that they hold also when $R$ is a strict complete intersection. A stronger result holds when $R$ is a standard graded ring and $F$ is a complex of graded $R$-modules. This will be dealt with in forthcoming work;  see also Corollary~\ref{co:chi-infinity} and \ref{ch:chi-graded}.

We begin with some general observations. For $i=0$ the inequality \eqref{eq:euler} reads
\[
\chi_R(F)\ge \beta^R_0(F)e(R)\,.
\]
Taking the sum over $i$ in \eqref{eq:euler} yields an inequality
\[
\chi_R(F) \ge \frac{\beta^R(F)}{2^d} e(R)\,.
\]
The third author~\cite{Walker:2017} proved that when $R$ is complete intersection with $\chr k\ne 2$, or when $R$ is Cohen-Macaulay and  contains a field of positive characteristic $p\ge 3$, then any $F\in \perf R$ with $\hh F\ne 0$ satisfies
\begin{equation}
\label{eq:walker}
\beta^R(F) \ge 2^{d} \frac{|\chi_R(F)|}{\sum_i \length_R \HH iF}\,.
\end{equation}

For now we record a simple example where these inequalities hold.

\begin{example}
\label{ex:dim1}
 If $\dim R \le 1$, then \eqref{eq:euler} holds for all short complexes in $\perf R$.

Indeed, we can assume any such $F$ is minimal, and so has the form
\[
0 \lra R^a \xra{\ \phi\ } R^a \lra 0\,.
\]
with $\phi(R^a)\subseteq \fm R^a$. With $\det(\phi)$ denoting the determinant of $\phi$,  from \cite[Lemma A.2.6]{Fulton:1998} one gets the first equality below
\[
\chi(F) = \chi( K(\det(\phi);R)) = e(\det(\phi), R)\ge e(\fm^a, R) = a \cdot e(R)\,.
\]
The  second equality is from \eqref{eq:serre}. The inequality holds because $\det(\phi)$ is in  $\fm^a$.
\end{example}

It will be helpful to consider all finitely generated $R$-modules $U$ with the property that for any short complex $F$ in $\perf R$ and integer $0\le i\le d$ one has
\begin{equation}
\label{eq:eulerU}
\binom di \chi_R(F\otimes_RU) \ge \beta^R_i(F)  e_d(U)\,.
\end{equation}
Observe that \eqref{eq:euler} is the case $U=R$. Here is a simple observation.

\begin{lemma}
\label{le:filtration}
Let $F$ be a short complex in $\perf R$.  If an $R$-module $W$ admits a finite filtration with subquotient modules satisfying \eqref{eq:eulerU} for $F$, then so does $W$.
\end{lemma}

\begin{proof}
Let $U_1,\dots, U_n$ be the subquotients of the filtration in the hypothesis. Evidently $\dim_RW \ge \dim_R U_j$ for each $j$, so the additivity of multiplicity   \cite[Corollary~4.7.7]{Bruns/Herzog:1998} yields the last equality below:
\begin{align*}
\binom di \chi(F,W)
	&= \sum_i \binom di  \chi_R(F,U_j) \\
	& \ge \sum_j \beta^R_i(F)  e_d( U_j)\\
	& = \beta^R_i(F)  e_d(W)\,.
\end{align*}
The first equality is by the additivity of Euler characteristics and the inequality holds by hypothesis.
\end{proof}

The proof suggests working in the Grothendieck group on $R$. We pick up on this theme in the ensuing sections. Next we introduce a class of modules for which the inequalities in \eqref{eq:eulerU} hold.

\subsection*{Ulrich modules}
For any  finitely generated maximal  Cohen-Macaulay $R$-module $M$ there is an inequality
\[
e_d(M) \ge \nu_R(M)\,.
\]
See, for example, \cite[Proposition 1.1]{Brennan/Herzog/Ulrich:1987}.  An \emph{Ulrich module} is a maximal Cohen-Macaulay $R$-module $U$ with $e_d(U)=\nu_R(U)$.  When the residue field of $R$ is infinite, this condition is equivalent to: there exists  a system of parameters $\bsx$ for $R$ with
\[
\kos{\bsx}U \simeq U/\fm U\,;
\]
in particular, $\bsx U = \fm U$. In fact, when $U$ is Ulrich any generic choice of $\bsx$ has these properties.  This characterization becomes helpful once we make the following observation which is immediate from the discussion in \ref{ch:gonflament}.

\begin{chunk}
\label{ch:Ulrich-basechange}
If $U$ is an Ulrich $R$-module and $(R,\fm) \to (S,\fn)$  is a  flat local map such that $\fm S =\fn$,  then the $S$-module $S\otimes_RU$ is Ulrich.
\end{chunk}

\begin{proposition}
\label{pr:ulrich-good}
Ulrich modules satisfy \eqref{eq:eulerU} for all short complexes in $\perf R$.
\end{proposition}

\begin{proof}
Let $U$ be an Ulrich module and $F$ a short complex in $\perf R$.  Given \ref{ch:gonflament} and \ref{ch:Ulrich-basechange}, we can inflate the residue field of $R$ if necessary to ensure that it is infinite. There then exists a sequence   $\bsr\colonequals r_1, \dots, r_d$, where $d=\dim R$,   such that the sequence $\bsr$ is regular on $U$ and  $U/\bsr U\cong k^{e(U)}$. One has quasi-isomorphisms
\[
\kos{\bsr}{F\otimes_RU}  \simeq F \otimes_R \kos{\bsr} U \simeq F\otimes_R k^{e(U)} \simeq (F\otimes_Rk)^{e(U)}
\]
Therefore one gets equalities
\[
 \length_R\HH i{\bsr; F\otimes_RU} = \beta^R_i(F)e(U)\,.
 \]
On the other hand, since $U$ is maximal Cohen-Macaulay $H_i(F\otimes_RU)=0$ for $i\ge 1$, as discussed~\ref{ch:short}, so that one has quasi-isomorphisms
\[
F\otimes_R U\simeq \HH 0{F\otimes_R U} \cong \HH 0F \otimes_RU\,.
\]
This yields a quasi-isomorphism
\[
\kos{\bsr}{F\otimes_RU} \simeq \kos{\bsr}{\HH 0F\otimes_RU}\,.
\]
From these computations we get
\begin{align*}
\beta^R_i(F)e(U)
	&= \length_R \HH i{\bsr; F\otimes_RU} \\
	&= \length_R \HH i{\bsr; \HH 0F\otimes_RU} \\
	& \leq \binom di \length_R(\HH 0F \otimes_RU) \\
	& = \binom di \chi_R(F\otimes_RU) \,.
\end{align*}
This is the desired inequality.
   \end{proof}

A local ring $(R,\fm)$ is a \emph{strict complete intersection} if its associated graded ring $\mathrm{\gr}_{\fm}(R)$ is complete intersection.  For example, every hypersurface is a strict complete intersection. Another example is the localization of a standard graded complete intersection at its homogenous maximal ideal.

\begin{theorem}
\label{th:sci}
When $R$ is a strict complete intersection local ring, any nonzero finite length $R$-module $M$ of finite projective dimension satisfies
\[
\binom di \length_R(M) \ge \beta^R_i(M)  e(R) \quad\text{for each $0\le i\le d$}\,.
\]
In particular, $\length_R(M) \ge (\beta^R(M)/2^d)\cdot e(R)$.
\end{theorem}

\begin{proof}
The hypothesis on $R$ implies that there exists a positive integer $s$ such that the free module $R^s$ has a finite filtration whose subquotients are Ulrich modules; this is implicit in \cite[\S2]{Backelin/Herzog/Ulrich:1991}; see also \cite[Theorem~V.28]{Ma:2014}.  Thus Proposition~\ref{pr:ulrich-good} and Lemma~\ref{le:filtration} imply that \eqref{eq:eulerU} holds for $U\colonequals R^s$  and the minimal free resolution $F$ of $M$.  Given the equalities
\[
\chi_R(F\otimes_R R^s) =s \cdot \length_R(M) \qquad\text{and}\qquad e(R^s)=s\cdot e(R)\,,
\]
the desired result follow.
\end{proof}

Next we prove the following result about perfect modules. This is connected to Lech's conjecture; we postpone that discussion to Section ~\ref{se:lech}.

\begin{proposition}
\label{pr:length-perfect}
Let $R$ be a local ring with infinite residue field.  If each short complex  in $\perf R$ satisfies \eqref{eq:euler} for $i=0$, then any perfect $R$-module $M$ satisfies
\[
e(M) \geq \nu_R(M) \cdot e(R)\,.
\]
If \eqref{eq:euler} holds for all $i$, then
\[
e(M) \geq \frac{\beta^R(M)}{2^{d-\dim_R(M)}} \cdot e(R)\,.
\]
\end{proposition}

\begin{proof}
Set $c\colonequals \dim_R(M)$.  Since the residue field of $R$ is infinite, there exists  a system of parameters $\bsr\colonequals r_1,\dots,r_c$ for $M$ such
that $e(M) = \chi(K\otimes_RM)$, where  $K$ is the Koszul complex on $\bsr$; see \eqref{eq:serre}. Let $G$ be a minimal free resolution of $M$. Since $M$ is perfect the length of $G$ equals $d-c$. It follows that $K\otimes_RG$ is a short complex in $\perf R$. Thus the hypothesis yields the inequality below:
\[
e(M) = \chi(K\otimes_RG) \ge \nu_R (\HH 0{K\otimes_RG})\cdot e(R) = \nu_R(M) e(R)\,.
\]
The equality on the left holds because $K\otimes_RG\simeq K\otimes_RM$, whereas the equality on the right holds because $\HH 0{K\otimes_RG}=M/\bsr M$. This justifies the first claim.

In the same vein, if \eqref{eq:euler} holds for each $i$ then
\[
\binom di e(M) \ge \beta^R_i(K\otimes_RG )e(R)\,.
\]
Summing these gives $2^d e(M) \ge \beta^R(K\otimes_RG)e(R)$. Since
\[
\beta^R(K\otimes_RG) =\beta^R(K)\beta^R(G) = 2^c \beta^R(M)
\]
the desired inequality follows.
\end{proof}

\begin{chunk}
We have already noted that the inequalities in \eqref{eq:euler} do not hold for all local rings. They can also fail if  we allow  complexes that are not short.  For example, let $G$ be any short complex in $\perf R$ with $\hh G\ne 0$,  choose an element $r\in \fm$, and let $F$ be the mapping cone of the map $G\xra{r} G$. It is easy to see that $F$ is a finite free complex of length $d+1$ satisfying $\hh F\ne 0$ and $\chi(F)=0$.

In a different direction, it is of interest to consider versions of \eqref{eq:euler} where the Euler characteristic of $F$ is replaced by the length of its homology module, $\hh F$. This becomes relevant only when $R$ is not Cohen-Macaulay. We have not much to say about this at the moment, except that here too the corresponding inequalities can fail for complexes that are not short; see \cite{Iyengar/Walker:2018} for similar phenomena. 

 Indeed, if $R$ is a complete intersection  of  codimension $c$  and $K$ the Koszul complex on a minimal list of generators of  $\fm$, then  $\HH 1 K \cong k^c$ and there is an isomorphism of $k$-algebras
\[
\hh K \cong \Lambda_k \shift \HH 1K \,;
\]
see \cite[Theorem~2.3.9]{Bruns/Herzog:1998}. Therefore $\length_R \hh K= 2^c$. On the other hand the multiplicity of $R$ can be much larger; for example, when
\[
R\colonequals k[x_1,\dots,x_n]/(x_1^{d_1},\dots, x_c^{d_c})
\]
one has $e(R) = d_1\cdots d_c$, which can be arbitrarily larger than $2^c$.
\end{chunk}

\section{Grothendieck groups}
\label{se:grog}
In this section we introduce Grothendieck groups of modules and related constructions. These  play a crucial role in the remainder of this work. As before  $(R,\fm,k)$ will be a local ring and $d\colonequals \dim R$.

\begin{chunk}
\label{ch:pairing}
Let $\grog R$ be the Grothendieck group of finitely generated $R$-modules. Let $\knot R$ be the Grothendieck group of  the category of finite free complexes with finite length homology, $\perf R$, modulo the exact complexes; see \cite{Gillet/Soule:1987}.   Given a finitely generated $R$-module $U$ and a complex $L$ in $\perf R$, set
\[
\pair{L}{U} \colonequals \chi (L\otimes_RU) = \sum_i (-1)^i \length_R \HH i{L\otimes_RU}\,.
\]
It induces a  pairing on Grothendieck groups
\[
\pair{-}{-} \colon \knot R  \otimes_{\bbZ} \grog R \lra \bbZ\,.
\]
We write   $\rgrog R$ for $\grog R$ modulo the subgroup of those classes $\alpha$ in $\grog R$ such that $\pair{-}{\alpha} = 0$ on  $\knot R$. This is the Grothendieck group of $R$ \emph{modulo numerical equivalence}. Likewise we write $\rknot R$ for the quotient of $\knot R$  by classes $\beta$ for which $\pair{\beta}{-} = 0$.  Abusing notation a bit we write
\[
\pair{-}{-}\colon \rknot R  \otimes_{\bbZ} \rgrog R  \lra \bbZ
\]
also for the induced pairing.

For an abelian group $A$ we set $A_\bbR  = A \otimes_\bbZ \bbR$.  In the sequel it will be convenient to work with the $\bbR$-vector spaces
\[
\euler R \colonequals {\rgrog R}_{\bbR}\qquad\text{and}\qquad
\eulerk R\colonequals {\rknot R}_{\bbR} \,.
\]
One has  an induced pairing of $\bbR$-vector spaces:
\[
\pair{-}{-}  \colon  \eulerk R  \otimes_{\bbR} \euler R \lra \bbR\,.
\]
Each of the induced maps
\[
\euler R \to \hom_{\bbR}(\eulerk R,\bbR) \qquad\text{and}\qquad
\eulerk R \to \hom_{\bbR}(\euler R,\bbR)
\]
is injective by construction. In particular, the pairing is perfect when the $\bbR$-vector space $\euler R$ is finite dimensional.  Kurano~\cite[Theorem~3.1]{Kurano:2004} has proved that $\euler R$ is finite dimensional whenever $R$ contains $\bbQ$ as a subring, or is essentially of finite type over a field,  $\bbZ$, or a complete discrete valuation ring. This applies in particular to the rings we consider in Section~\ref{se:lim-Ulrich-sheaves} onwards.
\end{chunk}

In what follows, the following set of primes in $R$ will play a crucial role:
\[
\Lambda(R)\colonequals \{\fp\in\spec R\mid \dim(R/ \fp)=\dim R\}\,.
\]
The assignment $[M]\mapsto (\length_{R_\fp}M_\fp)_{\fp\in \Lambda(R)}$ induces an $\bbR$-linear map
\begin{equation}
\label{eq:g0}
{\grog R}_{\bbR} \lra \bbR^{\Lambda(R)}
\end{equation}
This map is onto since it sends the collection $\{[R/\fp]\}_{\fp\in\Lambda(R)}$ to a basis of $\bbR^{\Lambda(R)}$. Moreover the kernel is generated by classes of the form $[R/\fp]$ with $\dim(R/\fp) < d$.

The following result is implicit in the proof of \cite[Proposition~3.7]{Kurano:2004}.

\begin{proposition}
\label{pr:representable}
For any $\bbR$-linear map $\gamma\colon \bbR^{\Lambda(R)} \to \bbR$, the composition of the map \eqref{eq:g0} with $\gamma$ coincides with $\pair{\alpha}{-}$ for some class $\alpha$ in ${\knot R}_{\bbR}$.
\end{proposition}

\begin{proof}
Let  $\fp_1,\dots,\fp_l$ be the element in $\Lambda(R)$. Pick elements $x_1, \dots, x_l$ such that $x_i \not\in \fp_i$ and $x_i \in \fp_j$ for all $j \ne i$.  Evidently the element $y_1\colonequals \sum_i x_i$ is not in any of the $\fp_j$ and so we may extend it to a system of parameters $y_1,\dots, y_d$ for $R$. For any positive integers $\bss\colonequals s_1, ..., s_l$, the elements
\[
{\bsx}_{\bss}, y_2,  \dots,  y_d \qquad\text{where ${\bsx}_{\bss}= \sum_i x_i^{s_i}$,}
\]
also form a system of parameters for $R$, as can be verified by going modulo $\fp_j$ for each  $j$. Let $K_{\bss}$ be the Koszul complex on the sequence above and consider the map
\[
\pair{K_{\bss}}{-}\colon {\grog R}_{\bbR}\lra \bbR\,.
\]
As with any Koszul complex on a system of parameters,  this map kills $[R/\fp]$ whenever $\dim(R/p) < d$ and hence it factors through the map \eqref{eq:g0}. Thus it determines a map $\bbR^{\Lambda(R)} \lra \bbR$ given by the tuple of integers
\[
v_{\bss} \colonequals ( \pair{K_{\bss}}{R/\fp_1},\dots, \pair{K_{\bss}}{R/\fp_l} )\,.
\]
We verify that the collection of all such tuples $v_{\bss}$, as the $\bss$ vary over $\bbZ_{\geqslant 1}^l$, span $\bbR^l$. Note that $v_{1,\dots, 1}$ is a tuple of strictly positive integers since $\pair{K_{1,\dots,1}}{-}$ computes multiplicity with respect to an $\fm$-primary ideal. Suppose
\[
v_{1,\dots,1} = (a_1,\dots, a_l) \qquad \text{with $a_i\ge 1$.}
\]
For any $i$  the sequence  ${\bsx}_{\bss}, y_2,  \dots,  y_d$ is $x_i^{s_i}, y_2,\dots y_d$ modulo $\fp_i$, so one has
\[
\pair {K_{\bss}}{R/\fp_i} = s_i \pair{K_{1,\dots,1}}{R/\fp_i}\,.
\]
Therefore  $v_{\bss} = (s_1 a_1,\dots, s_d a_d)$.  The result now follows easily.
\end{proof}

Here is one application of the preceding result.

\begin{corollary}
\label{co:representation}
The map \eqref{eq:g0} factors through $\euler R$.
\end{corollary}

\begin{proof}
If $\beta \in {\grog R}_{\bbR}$ is numerically equivalent to $0$ then Proposition~\ref{pr:representable} proves that $\beta$ goes to zero under the composition of \eqref{eq:g0} with any map $\gamma\colon \bbR^{\Lambda(R)}\to \bbR$ and hence it maps to $0$ under \eqref{eq:g0} itself.
\end{proof}

The next application  concerns the multiplicity $e_d(-)$ from \ref{ch:ed-symbol}.

\begin{corollary}
\label{co:ed-symbol}
The map $e_d(-)\colon {\grog R}_{\bbR}\to \bbR$ coincides with $\pair {\alpha}-$ for some class $\alpha$ in ${\knot R}_{\bbR}$. In particular it induces  a map of $\bbR$-vector spaces
\[
e_d(-)\colon \euler R \lra \bbR\,.
\]
\end{corollary}

\begin{proof}
When the residue field of $R$ is infinite,  the Koszul complex on a minimal  generating set for a minimal reduction of $\fm$ does the job; see \eqref{eq:serre}. In general, the local expression of the multiplicity~\cite[Corollary~4.7.8]{Bruns/Herzog:1998} means that  $e_d(-)$ factors through the map \eqref{eq:g0}, so Proposition~\ref{pr:representable} gives the desired result. Given this, Corollary~\ref{co:representation} implies that $e_d(-)$ factors through $\euler R$.
\end{proof}

These observations and constructions allow one to extend the proof of Lemma~\ref{le:filtration} to deduce the result below.

\begin{theorem}
\label{th:UlrichspanR}
If $[R]$ is a positive $\bbR$-linear combination in $\euler R$ of classes of Ulrich modules, then \eqref{eq:euler} holds for all short complexes in $\perf R$. \qed
\end{theorem}

Here are some simple examples where the hypothesis of the preceding result hold. The existence of Ulrich modules in these situation was recorded already in \cite{Brennan/Herzog/Ulrich:1987}; our only contribution is to observe that  they generate the class of the ring in the Gronthendieck group.

\begin{example}
\label{ex:dim1U}
The hypothesis of Theorem~\ref{th:UlrichspanR} holds when $\dim R\le 1$, and also when $R$ is a Cohen-Macaulay local ring of  minimal multiplicity.

Indeed if $\dim R=0$ the residue field $k$ is evidently an Ulrich module and $[R]=\length(R)\cdot [k]$ in $\grog R$.

Suppose $\dim R=1$ and set $S\colonequals R/\varGamma_{\fm}R$. This is a one-dimensional Cohen-Macaulay local ring, with maximal ideal $\fn\colonequals \fm S$. The  ideal $\fn^{e(S)-1}$ is an Ulrich  module for $S$, by  \cite[Lemma~2.1] {Brennan/Herzog/Ulrich:1987}, and hence also for $R$; this can be verified easily. Since $\pair{-}k=0$ on $\knot R$ we get that $[k]=0$ in $\euler R$ and hence, viewing $S/\fn^{e(R)-1}$ an $R$-module, one gets that $[S/\fn^{e(R)-1}]=0$. It follows that $[\fn^{e(R)-1}] = [S]$ in $\euler R$. It remains to note that $[\varGamma_{\fm}R]=0$ as well, so $[S]=[R]$, as desired.

Let $R$ be Cohen-Macaulay of minimal multiplicity; by the discussion above, we can assume $d\ge 1$. The $R$-module  $U\colonequals \syz^d(k)$ is an Ulrich $R$-module; see  \cite[Proposition 2.5]{Brennan/Herzog/Ulrich:1987}. By the definition of $U$ there is an exact sequence
\[
0\lra U\lra R^{n_{d-1}}\lra \cdots\lra R^{n_1}\lra R\lra k\lra 0.
\]
Since $\pair{-}k=0$ on $\knot R$ it follows that in $[U]=a[R]$ in $\rgrog R$ for some $a\ge 0$. It remains to observe that  $a\ge 1$ since the rank of $U$ is nonzero.
\end{example}

\begin{chunk}
\label{ch:even-hypersurface}
The hypothesis of Theorem~\ref{th:UlrichspanR} holds also when $R$ is a strict complete intersection, for in that case $R^s$ has a finite filtration whose subquotients are Ulrich modules; this fact was  key to proving Theorem~\ref{th:sci}. This raises the question whether for any complete intersection ring the class of Ulrich modules spans $[R]$ in $\euler R$. However, it is not even known  that such an $R$ has an Ulrich module.

If $R$ is a complete intersection of even dimension with an isolated singularity. It has been conjectured by Dao and Kurano~\cite[Conjecture~3.2(1)]{Dao/Kurano:2014} that $\euler R$ is one dimensional as a real vector space.  If this conjecture holds, then by Theorem~\ref{th:UlrichspanR}, the existence of a single (nonzero) Ulrich module $U$ implies that any short $R$-complex in $\perf R$ satisfies \eqref{eq:euler}.
\end{chunk}

\begin{chunk}
\label{ch:grop-topology}
We topologize the $\bbR$-vector space $\euler R$  by giving it the weakest topology for which every $\bbR$-linear map $\euler R \to \bbR$ is continuous, where $\bbR$ has the Euclidean topology.   This is the usual Euclidean topology when $\euler R$ is finite dimensional. We topologize  ${\grog R}_{\bbR}$ in the same way; with these topologies the quotient map  ${\grog R}_{\bbR}\to \euler R$ is continuous.

If $(\alpha_n)_{n\geqslant 0}$ is a convergent sequence in $\euler R$ (or in ${\grog R}_{\bbR}$) with limit $\alpha$, then
\[
\pair F{\alpha} = \lim_{n\to\infty} \pair F{\alpha_n} \qquad\text{for  $F$ in $\perf R$.}
\]
Our interest is in sequences with $\alpha_n = [M_n]/\nu_R(M_n)$, where the $M_n$ are nonzero finitely generated $R$-modules.
\end{chunk}

\begin{chunk}
\label{ch:frobenius}
Let $R$ be a local ring of positive characteristic $p$,  with perfect residue field, and such that the Frobenius endomorphism
\[
\vf\colon R\lra R
\]
is finite.  We write $\vf_*$ for restriction of scalars along $\vf$; the notation is in line with the one from algebraic geometry. Viewed as a functor on the category of  $R$-modules, $\vf_*$ has as left adjoint the base change functor
\[
\vf^*(M) \colonequals R^{\vf}\otimes_R M\,,
\]
where $R^{\vf}$ denotes  $R$ viewed as an $R$-$R$  bimodule with the canonical left action and  right action via $\vf$.

Since $\vf$ is finite $\vf_*$ induces a $\bbZ$-linear map on $\grog R$; we denote this also $\vf_*$. On the other hand, $\vf^*$ induces a $\bbZ$-linear map on ${\knot R}$, also denoted $\vf^*$. By the projection formula, for $\alpha\in \knot R$ and $\beta\in\grog R$ one has
\[
\pair {\vf^*(\alpha)}{\beta} = \pair {\alpha}{\vf_*(\beta)} \,.
\]
Given this adjunction it is clear that $\vf^*$ and $\vf_*$ induce maps
\[
\vf^* \colon \eulerk R \lra \eulerk R \qquad\text{and}\qquad \vf_* \colon \euler R\lra \euler R\,,
\]
and that these are adjoint to each other in the sense above.
\end{chunk}

\begin{chunk}
\label{ch:eigen-decomposition}
We remain in the framework of \ref{ch:frobenius}. For each integer $i\ge 0$ set
\[
{\grog R}_{\bbR}^{(i)} = \{\alpha\in{\grog R}_{\bbR}\mid \vf_*(\alpha) = p^i\alpha\}
\]
the eigenspace of $\vf_*$ with eigenvalue $p^i$.  The result below is implicit in \cite[\S2]{Kurano:1996}; see also \cite[Lemma~4]{Piepmeyer/Walker:2009}. The subset $\Lambda(R)\subseteq \spec R$ is as in \eqref{eq:g0}.

\begin{lemma}
\label{le:Feigen}
For $R$ as in \ref{ch:frobenius} there is an internal direct sum decomposition
\[
{\grog R}_{\bbR} = \bigoplus_{i=0}^d {\grog R}_{\bbR}^{(i)} \,.
\]
Moreover the assignment $[M]\mapsto (\length_{R_\fp}M_\fp)_{\fp\in \Lambda(R)}$ induces an isomorphism
\[
{\grog R}_{\bbR}^{(d)} \xra{\ \cong\ } \bbR^{\#\Lambda(R)}\,.
\]
\end{lemma}

We write $[M]_i$ for the image of $[M]$ in ${\grog R}_{\bbR}^{(i)}$. We will encounter modules $M$ that satisfy the following condition: There exists a \emph{nonzero} integer $m$ such that
\begin{equation}
\label{eq:rank}
\text{$\length_{R_\fp}(M_\fp) = m\cdot \length (R_\fp)$  for each $\fp\in\Lambda(R)$.}
\end{equation}
In this case $\dim M = \dim R$ and $e_d(M)= m\cdot e(R)$; see \cite[Corollary~4.7.8]{Bruns/Herzog:1998}.  Examples include modules $M$ of nonzero finite rank. When \eqref{eq:rank} holds one has
\begin{equation}
\label{eq:rank2}
[M]_d = \frac{e_d(M)}{e(R)} [R]_d\,.
\end{equation}
The result below will be used in the proof of Theorem~\ref{th:ma-convergence}.

\begin{lemma}
\label{le:frobenius-limit}
Let $R$ be as in \ref{ch:frobenius} and let $M$ be a finitely generated $R$-module. If $e_d(M)\ne 0$, then in ${\grog R}_{\bbR}$ one has
\[
\lim_{n\to\infty} \frac{\vf_*^n(M)}{e_d(\vf_*^n(M))} = \frac{[M]_d}{e_d(M)}\,.
\]
When in addition $M$ satisfies \eqref{eq:rank} the limit above equals $[R]_d/e(R)$.
\end{lemma}

\begin{proof}
One has $e_d(\vf_*^n(M))=p^{nd} e_d(M)$, since $e_d(-)$ factors through ${\grog R}_{\bbR}^{(d)}$; see the proof of Corollary~\ref{co:ed-symbol}. This fact and Lemma~\ref{le:Feigen} yield the first equality below:
\[
\lim_{n\to\infty} \frac{\vf_*^n(M)}{e_d(\vf_*^n(M))} = \lim_{n\to\infty}  \frac{[M]_0 +\cdots + p^{nd} [M]_d}{p^{nd} e_d(M)} = \frac{[M]_d}{e_d(M)}\,.
\]
For the second one, note that for any $\alpha$ in ${\grog R}_{\bbR}$ and sequence $(a_n)_{n\geqslant 0}$ of nonzero integers converging to $\infty$, the sequence $(\alpha/a_n)_{n\geqslant 0}$ converges to $0$ in ${\grog R}_{\bbR}$.

The last part of the statement is now immediate from \eqref{eq:rank2}.
\end{proof}
\end{chunk}

\begin{chunk}
\label{ch:euler-decomposition}
As noted earlier, the endomorphism $\vf_*$ on ${\grog R}_{\bbR}$ descends to $\euler R$. Thus the decomposition  in Lemma~\ref{le:Feigen} induces an analogous decomposition on $\euler R$. It follows from Corollary~\ref{co:representation} that the isomorphism in \emph{op.~cit.} induces isomorphisms
\[
{\grog R}_{\bbR}^{(d)} \xra{\ \cong\ } {\euler R}^{(d)} \xra{\ \cong \ } \bbR^{\#\Lambda(R)}\,.
\]
Another remark: While we have introduced these decompositions only in the context of \ref{ch:frobenius}, they exist in great generality, covering also rings not necessarily of positive characteristic, using the Riemann-Roch isomorphism; see~\cite{Fulton:1998}.
\end{chunk}

\begin{chunk}
\label{ch:Dutta}
Let $R$ be a local ring of positive characteristic $p$. The \emph{Dutta multiplicity} of a complex $F$ in $\perf R$ is
\[
\chi_{\infty}(F) \colonequals \lim_{n\to \infty} \frac{\chi((\vf^*)^n(F))}{p^{nd}}\,.
\]
When $R\to S$ is  flat map with $\fm S$ the maximal ideal of $S$, then it is easy to verify that $(\vf^*)^n(S\otimes_RF) \cong S\otimes_R (\vf^*)^n(F)$, and it follows that $\chi_\infty(F)=\chi_\infty(S\otimes_RF)$. Given this remark and \ref{ch:gonflament}, we can assume that $R$ is complete and $k$ is  algebraically closed when computing Dutta multiplicities. This puts us in the context of \ref{ch:frobenius},  and in this case the following expression for the Dutta multiplicity is well-known~\cite[\S5]{Roberts:1989}, and can be deduced easily from Lemma~\ref{le:frobenius-limit}.
\begin{equation}
\label{eq:duttam}
\chi_{\infty}(F)  = \pair F{[R]_d}\,.
\end{equation}
Since $\chi(F) = \pair F{[R]}$ the Dutta multiplicity is the Euler characteristic of $F$ whenever $[R]=[R]_d$ in $\euler R$. This is the case, for example, when $R$ is a complete intersection ring, and more generally when $R$ is numerically Roberts; see~\cite[\S6]{Kurano:2004}.
\end{chunk}

\section{Lim Cohen-Macaulay and lim Ulrich sequences}
\label{se:lus}
This section concerns the notion of a lim Ulrich sequence of modules introduced in \cite{Ma:2020}. Building on this, we introduce a notion of lim Ulrich point in the Grothendieck group modulo numerical equivalence. Our interest in such points is explained by Theorem~\ref{th:lus} that derives bounds on intersection multiplicities and Betti numbers from the existence of such points.

Throughout  $R$ will be a local ring and we set $d\colonequals \dim R$.

\begin{chunk}
\label{ch:lcm}
We say that a sequence $(U_n)_{n\geqslant 0}$ of finitely generated $R$-modules is a \emph{lim Cohen-Macaulay sequence} if each $U_n$ is nonzero, and for any short complex $F$ in $\perf R$ one has
\begin{equation}
\label{eq:lcm}
\lim_{n\to \infty} \frac {\length_R(\HH i{F\otimes_R U_n})}{\nu_R(U_n)}=0 \qquad \text{for $i\ge 1$.}
\end{equation}
The definition is from  \cites{Bhatt/Hochster/Ma:2100,Hochster:2017, Ma:2020}, with the caveat that in these sources the complexes $F$ are restricted to be Koszul complexes on systems of parameters. These are equivalent notions. In fact, in checking whether a given sequence of modules is lim Cohen-Macaulay it suffices to test that condition \eqref{eq:lcm} holds when $F$ is a Koszul complex on a \emph{single} system of parameters. This result is due to Bhatt, Hochster, and Ma~\cites{Bhatt/Hochster/Ma:2100,Hochster:2017}, and appears as Lemma~\ref{le:lcm} below.
\end{chunk}

The  observation below will be useful in the proof of Lemma~\ref{le:lcm} and in many of the later arguments.

\begin{lemma}
\label{le:lim-zero}
Let $(W(n))_{n\geqslant 0}$ be a sequence of $R$-complexes, $(v_n)_{n\geqslant 0}$ a sequence of positive integers, and $s$ an integer such that  \[
\lim_{n\to\infty} \frac{\length_R \HH i{W(n)}}{v_n}=0 \qquad\text{for all $i\ne s$.}
\]
Then for any finite free complex $P$ and integer $i$ the following equality holds:
\[
\limsup_{n\to \infty}  \frac{\length_R \HH i{P\otimes_R W(n)}}{v_n}= \limsup_{n\to \infty}  \frac{\length_R \HH {i-s} {P\otimes_R \HH s{W(n)}}}{v_n}\,.
\]
The corresponding equality involving $\liminf$ also holds.
\end{lemma}

In what follows it will be expedient to use the shorthand
\begin{equation}
\label{eq:h-notation}
h_i(X)\colonequals  \length_R \HH iX
\end{equation}
for any $R$-complex $X$ and integer $i$.

\begin{proof}[Proof of Lemma~\ref{le:lim-zero}]
The argument is a reduction to the following special case: If the limit of the sequence $(h_i(W(n))/v_n)_{n\geqslant 0}$ is $0$ for each $i$, then
\[
\lim_{n\to \infty} \frac{h_i(P\otimes_RW(n))}{v_n}= 0 \quad \text{for all $i$.}
\]
This can be verified by a simple induction on the number of nonzero terms in $P$.

Set $W'(n)\colonequals \tau_{\geqslant s} W(n)$, the good truncation of $W(n)$ below $s$; see \cite[1.2.7]{Weibel:2013}. Consider the induced  exact sequence of complexes
\begin{equation}
\label{eq:lim-zero}
0\lra  W'(n) \lra W(n) \xra{\ \pi\ } W''(n) \lra 0\,.
\end{equation}
By construction $\HH i{\pi}$ is an isomorphism in degrees $i\le s-1$ and $\HH i{W''(n)}=0$ for $i\ge s$. Thus hypotheses yields
\[
\lim_{n\to \infty} \frac{h_i(W''(n))}{v_n} = 0 \quad\text{for each $i$.}
\]
Then the already establish part of the result implies that
\[
\lim_{n\to\infty} \frac{h_i(P\otimes_RW''(n))}{v_n}=0 \quad\text{for each $i$.}
\]
Tensoring \eqref{eq:lim-zero} with $P$ and taking homology yields for each $i$ an exact sequence
\[
\HH {i+1}{P\otimes_R W''(n)} \to \HH {i}{P\otimes_R W'(n)}  \to \HH {i}{P\otimes_R W(n)} \to \HH {i}{P\otimes_R W''(n)}
\]
It follows from these computations that
\begin{equation}
\label{eq:lim-zero2}
\limsup_{n\to\infty}  \frac{h_i(P\otimes_RW'(n))}{v_n}=  \limsup_{n\to\infty} \frac{h_i(P\otimes_RW(n))}{v_n}
\end{equation}
for each integer $i$.  Observe also that by construction of \eqref{eq:lim-zero} one has
\[
\HH i{W'(n)}\cong
\begin{cases}
0 & \text{for $i<s$}\\
\HH i{W(n)} &\text{for $i\ge s$.}
\end{cases}
\]
Given this information and the equality \eqref{eq:lim-zero2} we can replace $W(n)$ by $W'(n)$ and  assume the $W(n)_i=0$ for $i<s$ and all $n$.  Then consider the exact sequence
\[
0\lra V(n) \lra W(n) \lra \shift^s \HH s{W(n)} \lra 0
\]
where the map on the right is the canonical surjection and $V(n)$ is its kernel.   Thus
\[
\HH i{V(n)}\cong
\begin{cases}
0 & \text{for $i=s$}\\
\HH i{W(n)} &\text{for $i\ne s$.}
\end{cases}
\]
The hypotheses  implies that the limit of the sequence $(h_i(V(n))/v_n)_{n\geqslant 0}$ is $0$ for each $i$. Then the already establish part of the result implies that
\[
\lim_{n\to\infty} \frac{h_i(P\otimes_RV(n))}{v_n}=0 \quad\text{for each $i$.}
\]
Given this one can argue as before to deduce the first equality below.
\begin{align*}
\limsup_{n\to\infty}  \frac{h_i(P\otimes_RW(n))}{v_n}
 &=  \limsup_{n\to\infty} \frac{h_i(P\otimes_R \shift^s \HH s{W(n)})}{v_n} \\
&=  \limsup_{n\to\infty} \frac{h_{i-s}(P\otimes_R \HH s{W(n)})}{v_n}
\end{align*}
This is the desired equality involving $\limsup$; the one for $\liminf$ can be verified in exactly the same way.
\end{proof}

Here is the promised result about lim Cohen-Macaulay sequences, from \cite{Bhatt/Hochster/Ma:2100}. We give a proof for the benefit of the first and third authors.

\begin{lemma}
\label{le:lcm}
A sequence $(U_n)_{n\geqslant 0}$ of finitely generated $R$-modules is lim Cohen-Macaulay if \eqref{eq:lcm} holds for the Koszul complex on a single system of parameters for $R$.
\end{lemma}

\begin{proof}
Suppose \eqref{eq:lcm} holds for $\kos{\bsr}R$, the Koszul complex on a system of parameters $\bsr\colonequals r_1,\dots,r_d$ for $R$. We first verify that it also holds $\kos{r_1^{a_1},\dots,r_d^{a_d}}R$ for any integers $a_i\ge 1$. To this end it suffices to verify that if \eqref{eq:lcm} holds for $\kos{r_1^a, \bsr_{\geqslant 2}}R$ for some $a\ge 1$, then it also holds for $\kos{r_1^{a+1},\bsr_{\geqslant 2}}R$.

Let $\cat D$ denote the derived category of $R$. The composition of maps $R\xra{r_1^a} R\xra{r_1}R$ induces an exact triangle
\[
\kos{r_1^a}R \lra \kos {r_1^{a+1}}R \lra \kos {r_1}R
\]
in $\cat D$,  and hence  an exact triangle
\[
\kos{r_1^a,\bsr_{\geqslant 2}}{U_n} \lra \kos{r_1^{a+1},\bsr_{\geqslant 2}}{U_n} \lra \kos{\bsr}{U_n}\,.
\]
Thus for each integer $i$ there is an inequality
\[
\length_R \HH i{r_1^{a+1},\bsr_{\geqslant 2};U_n} \le \length_R \HH i{r_1^{a},\bsr_{\geqslant 2};U_n} + \length_R \HH i{\bsr;U_n}\,.
\]
The desired result follows.

Let  $F$ be an $R$-complex with $\length_R \hh F$ finite. In particular the kernel of the natural map $R\to \hom_{\cat D}(F,F)$ is primary to the maximal ideal of $R$. By the discussion in the previous paragraph one can find a system of parameters $\bsr$ in this kernel with the property that \eqref{eq:lcm} holds for  $K\colonequals K(\bsr)$.

Since $\bsr$ annihilates $\hom_{\cat D}(F,F)$ the natural map $K\otimes F\to \shift^d F$ of $R$-complexes has a section in $\cat D$. This fact does not require  $\bsr$ to be a system of parameters, only that it annihilates $\hom_{\cat D}(F,F)$, and  can be verified by an induction on the length of the sequence $\bsr$. The case of length one is immediate from the exact triangle
\[
F\xra{\ r_1\ } F \lra \kos {r_1}F \lra \shift F
\]
for since the map on the left is zero in $\cat D$, the triangle above splits, that is to say, the map on the right has a section. Replacing $F$ by $\kos {r_1} F$ and using  induction yields the desired statement. In this step one has to use the fact that $\bsr$ annihilates $\hom_{\cat D}(\kos {r_1}F,\kos {r_1}F)$; see \cite[\S3]{Apassov:1999}, or \cite[1.5.3]{Avramov/Iyengar/Miller:2006}.

From the preceding discussion we deduce that $\HH j{F\otimes_R U_n}$ is a direct summand of $\HH {j+d}{F\otimes_R K\otimes_R U_n}$ for each integer $i$. It thus suffices to verify that for any short complex $F$ in $\perf R$   and  $i\ge d+1$ one has
\[
\lim_{n\to \infty} \frac{h_i(F\otimes_R K\otimes_R U_n)}{\nu_R(U_n)} =0 \,.
\]
Here we are using the shorthand introduced in \eqref{eq:h-notation}. As $K$ satisfies \eqref{eq:lcm}, Lemma~\ref{le:lim-zero} applied with $W(n) = K\otimes_R U_n$   and $P=F$ yields
\[
\limsup_{n\to \infty} \frac{h_i(F\otimes_R K\otimes_R U_n)}{\nu_R(U_n)}
	= \limsup_{n\to \infty} \frac{h_i(F\otimes_R \HH 0{K\otimes_R U_n})}{\nu_R(U_n)}
\]
for each $i$. It remains to note that when $i\ge d+1$ the terms of the limit on the right hand side are all zero since $F$ is a short complex.
\end{proof}

Here is a simple consequence of the preceding result. The proof is straightforward so it is omitted;  see~\cite[Lemma 2.3]{Ma:2020}.

\begin{corollary}
\label{co:lcm-basechange}
If $(U_n)_{n\geqslant 0}$ be a lim Cohen-Macaulay sequence of $R$-modules and $R\to S$ is an  flat local map with artinian closed fiber, then the sequence of $S$-modules $(S\otimes_RU_n)_{n\geqslant 0}$ is lim Cohen-Macaulay. \qed
\end{corollary}

\begin{lemma}
\label{eq:lcm-inequality}
Let $(U_n)_{n\geqslant 0}$ be a lim Cohen-Macaulay sequence. One has inequality
\[
\liminf_{n\to \infty} \frac{e_d(U_n)}{\nu_R(U_n)} \ge 1\,.
\]
\end{lemma}

\begin{proof}
Given Corollary~\ref{co:lcm-basechange} and the fact that both $e_d(-)$ and $\nu_R(-)$ are unchanged when we extend the residue field, we can assume that the residue field of $R$ is infinite. Let $K$ be the Koszul complex on a minimal generating set for a minimal reduction of $\fm$. For each integer $n$ one has
\[
\length_R \HH 0{K\otimes_RU_n}\ge \nu_R(U_n)\,.
\]
This justifies the inequality below; the first equality is by \eqref{eq:serre}:
\[
\liminf_{n\to \infty} \frac{e_d(U_n)}{\nu_R(U_n)}
= \liminf_{n\to\infty} \frac{\chi(K\otimes_RU_n)} {\nu_R(U_n)}
= \liminf_{n\to\infty} \frac{\length_R \HH 0{K\otimes_RU_n}} {\nu_R(U_n)} \ge 1\,.
\]
The second equality is by the lim Cohen-Macaulay condition.
\end{proof}

The preceding result motivates the definition below from \cite{Ma:2020}.

\begin{chunk}
\label{ch:lus}
A sequence $(U_n)_{n\geqslant 0}$ of $R$-modules is a \emph{lim Ulrich sequence} if it is lim Cohen-Macaulay and satisfies
\begin{equation}
\label{eq:lus}
\lim_{n\to \infty} \frac{e_d(U_n)}{\nu_R(U_n)} = 1\,.
\end{equation}
This bundles two conditions in one: the limit should exist and  it should be $1$.

If a nonzero $R$-module $U$ is maximal Cohen-Macaulay, respectively, Ulrich, then the sequence with $U_n=U$ for all $n$ is lim Cohen-Macaulay, respectively, lim Ulrich. In particular lim Ulrich sequences exist when $\dim R\le 1$; see Example~\ref{ex:dim1U}. Moreover if $(U_n)_{n\geqslant 0}$ is a lim Ulrich sequence of $R$-modules and $(R,\fm) \to (S,\fn)$ is a flat local map with $\fm S=\fn$, then the sequence of $S$-modules $(S\otimes_RU_n)_{n\geqslant 0}$ is lim Ulrich. This is immediate from Corollary~\ref{co:lcm-basechange} and \ref{ch:gonflament}.
\end{chunk}

\begin{chunk}
\label{ch:limU-point}
An element of $\euler R$ is a \emph{lim Ulrich point} if it is in the strictly positive real cone spanned by points of the form
\[
\lim_{n\to \infty} \frac{[U_n]}{\nu_R(U_n)}\,,
\]
where the sequence $(U_n)_{n\geqslant 0}$ is lim Ulrich and the limit exists.  A \emph{lim Cohen-Macaulay point} has the obvious meaning.

For later use we record the observation that if $\euler R$ has a lim Cohen-Macaulay point and $R\to S$ is a flat map with artinian closed fiber, then $\euler S$ has a lim Cohen-Macaulay point; this is immediate from Corollary~\ref{co:lcm-basechange}. The same holds for lim Ulrich points if also $\fm S = \fn$, where $\fm$ and $\fn$ are the maximal ideals of $R$ and $S$, respectively.

If $(U_n)_{n\geqslant 0}$  is a lim Cohen-Macaulay sequence and the sequence $([U_n]/\nu_R(U_n))$ converges in $\euler R$, say to a point $\alpha$,  then
\begin{equation}
\label{eq:limU-point1}
\pair F{\alpha} =\lim_{n\to \infty} \frac{\chi(F, U_n)}{\nu_R(U_n)} = \lim_{n\to \infty} \frac{\length_R\HH 0{F\otimes_R U_n}}{\nu_R(U_n)}
\end{equation}
for any short complex $F$ in $\perf R$.  The first equality holds because of the convergence, whereas the second one is a consequence of \eqref{eq:lcm}.  It will also be useful to note that if the sequence $(U_n)$ is lim Ulrich then
\begin{equation}
\label{eq:limU-point2}
e_d(\alpha) = \lim_{n\to \infty} \frac{ e_d(U_n)}{\nu_R(U_n)} = 1\,.
\end{equation}
The second equality is  from \eqref{eq:lus}.

By the discussion above if $\alpha$ is a lim Cohen-Macaulay point in $\euler R$ then $\pair F{\alpha} >0$ for any short complex $F$ in $\perf R$ with $\hh F$ nonzero; in particular $0$ cannot be a lim Cohen-Macaulay point.
\end{chunk}

\begin{lemma}
\label{le:lcm-converges}
If $(U_n)_{n\geqslant 0}$ is a lim Cohen-Macaulay sequence, then for any short complex $F$ in $\perf R$   one has
\[
\length_R \HH 0F \ge \limsup_{n\to \infty}  \frac{\pair F{U_n}}{\nu_R(U_n)}\ge \liminf_{n\to \infty}  \frac{\pair F{U_n}}{\nu_R(U_n)}\ge  \nu_R \HH 0F \,.
\]
Thus if $\euler R$ is finite dimensional $(U_n/\nu_R(U_n))$ has convergent subsequence.
\end{lemma}

\begin{proof}
Since $F$ satisfies \eqref{eq:lcm} one gets an equality
\[
\limsup_{n\to \infty}  \frac{\pair F{U_n}}{\nu_R(U_n)} = \limsup_{n\to \infty}  \frac{h_0(F\otimes_RU_n)}{\nu_R(U_n)}\,,
\]
and similarly for the lim inf. It remains to observe that there are inequalities
\[
h_0(F)\nu_R(U_n)  \ge h_0(F\otimes_RU_n) \ge  \nu_R (\HH 0F)\nu_R(U_n) \,.
\]
These hold because $\HH 0{F\otimes_R U_n}=\HH 0F\otimes_R U_n$, and for any $R$-modules $H, U$ there are surjections
\[
H \otimes_R R^{\nu_R(U)}\twoheadrightarrow H \otimes_R U \twoheadrightarrow (H/\fm H)\otimes_k (U/\fm U)\,.
 \]
 This justifies the stated inequalities.
 
When  $\euler R$ is finite dimensional, its topology is the Euclidean one, and it is dual to $\eulerk R$. Since the classes of the complexes $F$ in $\perf R$ span $\eulerk R$,  the already established part of the result means that $(U_n/\nu_R(U_n))$ is  bounded in the Euclidean metric. Then the stated assertion is now clear.
\end{proof}

 The proof of the next result follows that of Proposition~\ref{pr:ulrich-good}.

\begin{theorem}
\label{th:lus}
Let $(R,\fm,k)$ be a local ring.  If $\alpha$ is a lim Ulrich point in $\euler R$, then for any short complex $F$ in $\perf R$   and integer $i$ there is an  inequality
\[
\binom di \pair F{\alpha} \ge \beta^R_i(F) e_d(\alpha)\qquad \text{where $d=\dim R$.}
\]
In particular $2^{d} \pair F{\alpha} \ge  \beta_R(F)  e_d(\alpha)$.
\end{theorem}

\begin{proof}
By definition, $\alpha$ is a  positive linear combination of  elements of the form
\[
\lim_{n\to \infty} \frac{[U_n]}{\nu_R(U_n)}
\]
with  $(U_n)$ a  lim Ulrich sequence.   Since both $\pair F{-}$ and $e_d(-)$ are additive on $\euler R$ we can suppose that $\alpha$ is an element of the form above.  Thus  by \eqref{eq:lus} the desired result is that
\[
\binom di \pair F{\alpha} \ge \beta^R_i(F) \,.
\]
Moreover, by the discussion in \ref{ch:lus}, expanding $k$ if necessary, we can pick a system of parameters $\bsr$ that generates a minimal reduction for $\fm$. Let $K$ be the Koszul complex on $\bsr$ so that $e_d(M)=\chi(K\otimes_RM)$ for any finitely generated $R$-module $M$; see \eqref{eq:serre}. We estimate the limiting behavior of the sequence
\[
\frac{h_i(K\otimes_R F\otimes_R U_n)}{\nu_R(U_n)}
\]
in two ways. These estimates are obtained by repeated application of  Lemma~\ref{le:lim-zero} with $v_n=\nu_R(U_n)$ and $s=0$, and different choices of $W(n)$ and $P$.

For one thing, since $(U_n)$ is a lim Cohen-Macaulay sequence we can apply it with $W(n) = F\otimes_R U_n$ to get the first equality below:
\begin{align*}
\limsup_{n\to \infty}  \frac{h_i(K\otimes_R F\otimes_R U_n)}{\nu_R(U_n)}
	&= \limsup_{n\to \infty}  \frac{h_i(K\otimes_R \HH 0{F\otimes_R U_n})}{\nu_R(U_n)} \\
	&\le \limsup_{n\to \infty}  \binom di \frac{h_0(F\otimes_R U_n)}{\nu_R(U_n)} \\
	& =  \binom di \pair F{\alpha} \,.
\end{align*}
The inequality holds because $K$ is a finite free complex, of rank $\binom di$ in degree $i$. The second equality is by \eqref{eq:limU-point1}.

On the other hand noting that $K\otimes_R F\otimes_RU_n\cong F\otimes_R K\otimes_RU_n$ as $R$-complexes and reversing the roles of $F$ and $K$ in the first step of the argument above yields
\begin{align*}
\limsup_{n\to \infty} \frac{h_i(K\otimes_R F\otimes_R U_n)}{\nu_R(U_n)}
	&= \limsup_{n\to \infty} \frac{h_i(F\otimes_R \HH 0{K\otimes_R U_n})}{\nu_R(U_n)}\\
	&= \limsup_{n\to \infty}  \frac{h_i(F\otimes_R U_n/\bsr U_n)}{\nu_R(U_n)}\\
\end{align*}
Consider the exact sequences
\begin{equation}
\label{eq:lus1}
0\lra \frac{\fm U_n}{\bsr U_n} \lra \frac{U_n}{\bsr U_n} \lra \frac{U_n}{\fm U_n}\lra 0
\end{equation}
From this sequence and \eqref{eq:lus} it follows that
\[
\lim_{n\to\infty} \frac{\length_R(\fm U_n/\bsr U_n)}{\nu_R(U_n)}=0\,.
\]
Therefore  Lemma~\ref{le:lim-zero}, now applied with $W(n) = \fm U_n/\bsr U_n$ and $P\colonequals F$ yields
\[
\limsup_{n\to\infty} \frac{h_i(F\otimes_R (\fm U_n/\bsr U_n))}{\nu_R(U_n)}=0
\]
for all $i$.  Feeding this information  back into  the exact sequence \eqref{eq:lus1}  gives for each integer $i$ the first equality below
\begin{align*}
\limsup_{n\to \infty} \frac{h_i(F\otimes_R U_n/\bsr U_n)}{\nu_R(U_n)}
	&= \limsup_{n\to \infty}  \frac{h_i(F\otimes_R U_n/\fm U_n)}{\nu_R(U_n)}\\
	&= \limsup_{n\to \infty}  \frac{h_i(\hh{F\otimes_Rk}\otimes_k{U_n/\fm U_n})}{\nu_R(U_n)}\\
	&= \beta^R_i(F)
\end{align*}
In summary one gets that
\[
\limsup_{n\to \infty} \frac{h_i(K\otimes_R F\otimes_R U_n)}{\nu_R(U_n)}  = \beta^R_i(F)\,.
\]
Combining this with the upper bound $\binom di \pair F{\alpha}$ for the limit obtained earlier yields the desired inequality.
\end{proof}

The corollary below, which extends Theorem~ \ref{th:UlrichspanR}, is our main motivation for considering lim Ulrich points.

\begin{corollary}
\label{co:lus}
Let $F$ be a short complex in $\perf R$. If $[R]$ is a lim Ulrich point in $\euler R$, then the inequalities in \eqref{eq:euler} hold for $\chi(F)$.

When the ring $R$ is as in \ref{ch:frobenius} and the class $[R]_d$ in $\euler R$ is a lim Ulrich point, then the analogues of \eqref{eq:euler} for $\chi_{\infty}(F)$ hold.
\end{corollary}

\begin{proof}
These assertions are immediate from Theorem~\ref{th:lus}, once we observe that $\chi(F)=\pair F{[R]}$; for the second we need that $\chi_{\infty}(F) = \pair F{[R]_d}$ from \eqref{eq:duttam}.
\end{proof}

Even if the class of $R$ is not a lim Ulrich point, just the existence of lim Ulrich points yields interesting estimates on Euler characteristics of short complexes.

\begin{corollary}
When $\euler R$ has a lim Ulrich point, for any short complex $F$ in $\perf R$   and integer $i$ there is an inequality
\[
\binom {d}i \length_R \HH 0F \ge\beta^R_i(F) \qquad\text{where $d=\dim R$.}
\]
In particular $2^{d} \length_R \HH 0F \ge \beta^R(F)$.
\end{corollary}

\begin{proof}
This is immediate from Lemma~\ref{le:lcm-converges}  and Theorem~\ref{th:lus}, given \eqref{eq:limU-point2}.
\end{proof}

Combining the inequality in the result above with \eqref{eq:walker}, when it applies, yields
\[
2^{d} \length_R \HH 0F \ge \beta^R(F) \ge 2^d \frac{|\chi_R(F)|}{\length_R\hh F}\,.
\]
When $F$ is a free resolution of a nonzero finite length $R$-module $M$, this reads
\[
2^d\length_R(M) \ge \beta^R(F) \ge 2^d\,.
\]

 We record another application of Lemma~\ref{le:lim-zero},  which recovers \cite[Lemma~2.9]{Ma:2020}, and leads to the material presented in the next section.  While the hypothesis is rather stringent, in the graded context one can run the  argument given below degreewise yielding a much stronger result. This is the essential idea in the proof of Theorem~\ref{th:usheaf-umodule}, which is why we bring this up.

\begin{lemma}
\label{le:lch-lcm}
If $(U_n)_{n\geqslant 0}$ is a sequence of nonzero $R$-modules such that for each integer $i\le d -1$ one has
\[
\lim_{n\to \infty} \frac{\length_R  \mathrm{H}_{\fm}^i(U_n)}{\nu_R(U_n)} = 0
\]
then the sequence $(U_n)$ is lim Cohen-Macaulay.
\end{lemma}

\begin{proof}
For each $F$  in $\perf R$ one has a quasi-isomorphism
\[
F\otimes_R U_n\simeq F\otimes_R \mathrm{R}\Gamma_{\fm}(U_n)\,.
\]
Thus applying Lemma~\ref{le:lim-zero} with $W_n = \mathrm{R}\Gamma_{\fm}(U_n)$ and $v_n=\nu_R(U_n)$ yields for any integer $i$  an equality
\[
\limsup_{n\to \infty} \frac{h_i(F\otimes_R U_n)}{v_n} = \limsup_{n\to \infty} \frac{h_{i+d}(F\otimes_R \mathrm{H}_{\fm}^d(U_n))}{v_n}\,.
\]
 When $F$ is a short complex the homology modules on the right are zero whenever $i\ge 1$. This completes the proof.
\end{proof}

\begin{chunk}
\label{ch:graded-vs-local}
Let $k$ be a field and $A$ a \emph{standard graded} $k$-algebra, that is to say, $A_0=k$ and $A$ is finitely generated as a $k$-algebra by its component $A_1$ of degree one.  Let $\fm\colonequals A_{\geqslant 1}$ be the homogenous maximal ideal of $A$. Set $R\colonequals A_\fm$; this is a local ring with maximal ideal $\fm R$, and residue field $k$. Set $d\colonequals \dim R = \dim A$.

Given a finitely generated graded $A$-module $M$, we write $e_d(M)$ for the multiplicity with respect to the $\fm$. It can be computed as so:
\[
e_d(M)  = (d-1)! \lim_{n\to \infty} \frac{\rank_k(M_n)}{n^{d-1}}\,.
\]
In this context $\perf A$ will be assumed to consist of graded free modules, with differentials respective the grading, and homology of finite rank over $k$. A short complex is to have length $d$.  A sequence $(U_n)_{n\geqslant 0}$ of finitely generated graded $A$-module is lim Cohen-Macaulay if \eqref{eq:lcm} holds for each short complex $F$ in $\perf A$; the sequence is lim Ulrich when in addition \eqref{eq:lus} holds.

\begin{proposition}
\label{pr:graded-vs-local}
Let $(U_n)_{n\geqslant0}$ be a sequence of finitely generated graded $A$-modules.
\begin{enumerate}[\quad\rm(1)]
\item
If \eqref{eq:lcm} holds for the Koszul complex on a single homogenous system of parameters for $A$, then the given sequence  is lim Cohen-Macaulay.
\item
The given sequence is lim Cohen-Macaulay if and only if the sequence of $R$-modules $((U_n)_\fm)_{n\geqslant 0}$  is lim Cohen-Macaulay.
\item
The given sequence is lim Ulrich if and only if the sequence of $R$-modules $((U_n)_\fm)_{n\geqslant 0}$ is lim Ulrich.
\end{enumerate}
\end{proposition}

\begin{proof}
For any finitely generated graded $R$-module $M$ one has equalities
\[
\nu_A(M) = \nu_{R}(M_\fm) \qquad\text{and}\qquad e_d(M) = e_d(M_{\fm})\,.
\]
These facts will be used in the argument without further comment.

(2) Let $F$ be a short complex in $\perf A$. The $A$-modules $\HH i{F\otimes_A U_n}$ are $\fm$-torsion, so the natural localization map is an isomorphism:
\begin{equation}
\label{eq:graded-vs-local}
\HH i{F\otimes_AU_n} \xra{\ \cong\ } \HH i{F\otimes_AU_n}_\fm\cong \HH i{F_\fm \otimes_R (U_n)_\fm}\,.
\end{equation}
In particular the localization map $\HH iF\to \HH i{F_\fm}$ is an isomorphism, so that $F_\fm$ is a short complex in $\perf R$. It follows that if the sequence of $R$-modules $((U_n)_\fm)_{n\geqslant 0}$ is lim Cohen-Macaulay, then so is the sequence of $A$-modules $(U_n)_{n\geqslant 0}$. This settles the ``if" part in (2).

Let $\bsa$ be  a homogenous system of parameters for $A$; their images in $R$, which we also denote $\bsa$, are a system of parameters for $R$.
Given this fact and the isomorphisms \eqref{eq:graded-vs-local}, it follows that when \eqref{eq:lcm} holds for $(U_n)_{n\geqslant 0}$ and the Koszul complex $K(\bsa; A)$, then \eqref{eq:lcm} holds for  $((U_n)_\fm)_{n\geqslant 0}$ and the Koszul complex $K(\bsa;R)$. Then  Lemma~\ref{le:lcm} implies that the sequence $((U_n)_{\fm})_{n\geqslant 0}$ is lim Cohen-Macaulay. This settles the ``only if" part of (2) and, given the ``if" part, also (1).

(3) is now clear, given the equalities above.
\end{proof}
\end{chunk}

\section{Lim Ulrich sequences of sheaves}
\label{se:lim-Ulrich-sheaves}

Throughout this section $k$ is an infinite field, $A$ a standard graded $k$-algebra; see \ref{ch:graded-vs-local}. Set $d\colonequals \dim A$; we assume $d\ge 2$ to avoid  unimportant special cases. Set
\[
X\colonequals \proj A \qquad\text{and}\qquad c\colonequals \dim X = d-1\,.
\]
In this section we introduce lim Ulrich sequences of sheaves on $X$ and relate these to lim Ulrich sequences of $A$-modules. This is used to recover a construction, due to the second author, of lim Ulrich sequences on $A$, when $k$ has positive characteristic; see Theorem~\ref{th:ma-revisited}. We begin by recalling  some basic facts about sheaves on $X$.

\begin{chunk}
\label{ch:fac}
For any coherent sheaf $\calf$ on $X$ we set
\[
\varGamma_*(\calf) \colonequals \oplus_{t\in\bbZ} \CH 0X{\calf(t)}
\]
viewed as a graded $A$-module in the usual way. We shall be interested in coherent sheaves $\calf$ that satisfy the following condition:
\begin{equation}
\label{eq:depth-condition}
\CH 0X{\calf(t)}=0 \qquad \text{for $t\ll 0$}.
\end{equation}
This is equivalent to the condition that $\varGamma_*(\calf)$ is finitely generated as an $A$-module. The following alternative characterization that will be useful in the sequel:  When $\dim X\ge 1$,  a coherent sheaf $\calf$  on $X$ satisfies \eqref{eq:depth-condition} if and only if
\[
\depth_{\calo_{x}} \calf_x\ge 1\quad\text{for all closed points $x\in X$,}
\]
see \cite[Corollaire~1.4]{sga2:2005}. In the remainder of this section $\calf$ will be a sheaf satisfying condition~\eqref{eq:depth-condition}. Its \emph{Euler characteristic} is
\[
\chi(\calf)\colonequals \sum_i (-1)^i h^i(X,\calf) \qquad \text{where}\qquad h^i(X,\calf)\colonequals \rank_k \CH iX{\calf}\,.
\]
The \emph{Hilbert polynomial} of $\calf$ is a polynomial $h_{\calf}(z)$ with rational coefficients with
\[
h_{\calf}(t) = \chi(\calf(t)) \quad \text{for all $t\in \bbZ$.}
\]
This is a polynomial of degree at most $c$ and the coefficient in degree $c$ is $e_d(\varGamma_*(\calf))/c!$. Thus the $c$th difference of $h_{\calf}(t)$ is $e_d(\varGamma_*(\calf))$ so that
\[
e_d(\varGamma_*(\calf)) = \sum_{0\leqslant i,j\leqslant c} (-1)^{i+j} {\binom cj} h^i(X, \calf(t-j))\,.
\]
In particular since the righthand side is independent of $t$ we get
\begin{equation}
\label{eq:e-from-hilb}
e_d(\varGamma_*(\calf))= \sum_{0\leqslant i,j\leqslant c} (-1)^{i+j} {\binom cj} h^i(X, \calf(-j)) \,.
\end{equation}

\end{chunk}

\begin{chunk}
\label{ch:Usheaf}
An \emph{Ulrich sheaf} on $X$ is a nonzero coherent sheaf $\calf$ such that $h^i(X,\calf(t))=0$ except possibly when
\begin{equation}
\label{eq:u-range}
i=0 \text{ and } t\ge 0,\qquad\text{or}\qquad  i=c\text{ and } t\le -c-1\,.
\end{equation}
Sometimes it is convenient to view these conditions in terms of the total cohomology
\[
h^*(X,\calf(t))\colonequals \sum_ih^i(X,\calf(t)),
\]
and as function of $t$, namely: $\calf$ is Ulrich if and only if for each $t$ one has
\[
h^*(X,\calf(t)) =
\begin{cases}
h^0(X,\calf(t)) & \text{when $t\ge 0$}\\
h^c(X,\calf(t)) & \text{when $t\le -c-1$}\\
0 & \text{otherwise.}
\end{cases}
\]

When $X=\bbP^c_k$ a sheaf $\calf$ is Ulrich if and only if it is isomorphic to a direct sum of copies of $\calo_X$.   For a general $X$, let
\[
\pi\colon X\to \bbP^c_k
\]
be a finite \emph{linear} projection; the linearity condition is that $\pi^*\calo_{\bbP^c_k}(1)=\calo_X(1)$. Such a projection exists because $k$ is infinite. Then a coherent sheaf $\calf$ on $X$ is Ulrich precisely when $\pi_{*}\calf$ is an Ulrich sheaf on $\bbP^c_k$, that is to say, a direct sum of copies of $\calo_{\bbP^c_k}$. This is a direct consequence of the projection formula
\[
\CH iX{\calf(t)} \cong \CH i{\bbP^c_k}{\pi_*\calf(t)}\,.
\]
\end{chunk}

We propose the following definition.

\begin{chunk}
\label{ch:limUsheaves}
A \emph{lim Ulrich sequence} of sheaves on $X$ is a sequence $(\calf_n)_{n\geqslant 0}$ of coherent sheaves on $X$ for which the following properties hold:
\begin{enumerate}[\quad\rm(1)]
\item
$h^0(X,\calf_n)\ne 0$ for all $n\gg 0$;
\item
There exists an integer $t_0$ such that $h^0(X,\calf_n(t))=0$ for $t\le t_0$ and all $n$;
\item
There exists an integer $t_1$ such that $h^{\geqslant 1}(X,\calf_n(t))=0$ for $t\ge t_1$ and all $n$;
\item
Except possibly when $i=0$ and $t\ge 0$, or  $i=c$ and $t\le -c-1$,  one has
\[
\lim_{n\to\infty}\frac{h^i(X,\calf_n(t))}{h^0(X,\calf_n)}=0\,.
\]
\end{enumerate}
The range of values of $i$ and $t$ arising in (4) is precisely the one from \eqref{eq:u-range}.
\end{chunk}

The next result explains our interest in this notion; Lemma~\ref{le:lch-lcm} is a precursor.

\begin{theorem}
\label{th:usheaf-umodule}
Let $k$ be an infinite field, $A$ a standard graded $k$-algebra, and set $X\colonequals \proj A$. If $(\calf_n)_{n\geqslant 0}$ is a lim Ulrich sequence of sheaves on $X$,  then the sequence of graded $A$-modules $(\varGamma_*(\calf_n))_{n\geqslant 0}$ is  lim Ulrich.
\end{theorem}

\begin{proof}
 Set $U_n\colonequals \varGamma_*(\calf_n)$, viewed as a graded $A$-module. We write  $U_{n,t}$ for its component in degree $t$.   Condition \ref{ch:limUsheaves}(2) implies that $U_{n,t}=0$ for $t\le t_0$. In particular the $A$-module $U_n$ is finitely generated; see the discussion below \eqref{eq:depth-condition}.

 The lim Ulrich condition on the sequence $(U_n)_{n\geqslant 0}$ involves $\nu_A(U_n)$, whereas the lim Ulrich condition on $(\calf_n)_{n\geqslant 0}$ is in terms of $h^0(X,\calf_n)$, that is to say,  $\rank_k(U_{n,0})$. We begin by comparing these numbers: One has
 \begin{align*}
\nu_A(U_n)
	&= \rank_k \mathrm{Coker}(A_1\otimes_k U_n\to U_n) \\
	&\ge \rank_k \mathrm{Coker}( A_1\otimes_k U_{n,-1} \to U_{n,0}) \\
	&\ge  h^0(X,\calf_n) -\rank_k(A_1) h^0(X,\calf_n(-1)) \,.
\end{align*}
Thus from  \ref{ch:limUsheaves}(4) one obtains
\begin{equation}
\label{eq:nu-estimate}
\liminf_{n\to\infty} \frac{\nu_A(U_n)}{h^0(X,\calf_n)}\ge 1\,.
\end{equation}
This observation will be of use later on. At the end we prove that this limit is $1$.

We verify that the sequence $(U_n)_{n\geqslant 0}$ is lim Cohen-Macaulay by mimicking the argument from Lemma~\ref{le:lch-lcm}. The advantage will be that since we are in the graded case, we can focus on one (internal) degree at a time. To that end recall that $\lch i{\fm}{U_n}$,  the local cohomology of $U_n$ at the homogeneous maximal ideal $\fm\colonequals A_{\geqslant 1}$ of $A$, acquires a grading from $U_n$. We write $\lch i{\fm}{U_n}_t$ for the component in degree $t$. This can be computed in terms of the sheaf cohomology of $\calf_n$ as follows:
\[
\lch i{\fm}{U_n}_t \cong
\begin{cases}
0 & i=0,1 \\
 \CH {i-1}X{\calf_n(t)} & i\ge 2\,.
\end{cases}
\]
See, for example, \cite[Corollary~6.2]{Kunz:2008}, noting that since $U_n=\varGamma_*(\calf_n)$, the sheafification of $U_n$ is $\calf_n$.   Combining this observation with condition \ref{ch:limUsheaves}(3) yields
\begin{equation}
\label{eq:lch-vanishing}
\lch i{\fm}{U_n}_t =0 \qquad\text{for $t\ge t_1$ and all $i$ and $n$,} \\
\end{equation}
whereas  \ref{ch:limUsheaves}(4) yields
\begin{equation}
\label{eq:lch-lim-vanishing}
\lim_{n\to \infty} \frac{\rank_k \lch i{\fm}{U_n}_t}{h^0(X,\calf_n)} =0 \qquad\text{for $i\ne d$ and all $t$.}
\end{equation}
At this point we have hypotheses akin to those in Lemma~\ref{le:lch-lcm}. In the rest of the proof it will be convenient to extend the use of notation~\ref{eq:h-notation} to the graded context: Given a complex $W$ of graded $A$-modules, set
\[
h_{i,t}(W)\colonequals \rank_k \HH iW_t\,.
\]
Fix a homogeneous system of parameters  $\bsa\colonequals a_1,\dots,a_d$  of $A$ and let $K$ be the Koszul complex on this sequence.  For any complex $X$ of graded $A$-modules we write $h_{i,t}(\bsa; W)$ for $h_{i,t}(K\otimes_A W)$.

We verify \eqref{eq:lcm} for $F=K$, which entails  $(U_n)_{n\geqslant 0}$ is lim Cohen-Macaulay; see Lemma~\ref{le:lcm}.   As in Lemma~\ref{le:lch-lcm} we use the quasi-isomorphism
\[
\kos{\bsa}{U_n}\simeq \kos{\bsa}{\mathrm{R}\Gamma_{\fm}(U_n)}\,,
\]
to transfer information we have on local cohomology modules to Koszul homology. For a start, given  $U_{n,t}=0$ for $t\le t_0$ and \eqref{eq:lch-vanishing}, for each $n$ and $i$ one has
\begin{equation}
\label{eq:koszul-bound}
h_{i,t}(\bsa; U_n) = 0 \qquad\text{for $t\le t_0$ and $t\ge t_1 + |\bsa|$,}
\end{equation}
where $|\bsa|\colonequals |a_1|+\dots+|a_d|$. This can be proved by a simple spectral sequence argument or through a claim akin to the one below,  analogous to Lemma~\ref{le:lim-zero}. Caveat: $W(n)$ is the $n$th term of a sequence and is not to be confused with a twist.

\begin{claim}
Let $(W(n))_{n\geqslant 0}$ be a sequence of $A$-complexes of graded modules, $(v_n)_{n\geqslant 0}$ a sequence of nonzero integers, and $s$  an integer such that
 \[
\lim_{n\to\infty} \frac{h_{i,t}(W(n))}{v_n}=0 \qquad\text{for all $i\ne s$ and all $t$.}
\]
Then for all $i$ and $t$ one has
\[
\limsup_{n\to \infty}  \frac{h_{i,t}(\bsa; W(n))}{v_n}= \limsup_{n\to \infty}  \frac{h_{i-s,t}(\bsa; \HH s{W(n)})}{v_n}
\]
\end{claim}
The proof is similar to that of Lemma~\ref{le:lim-zero} and is omitted.  Given \eqref{eq:lch-lim-vanishing} we can apply the claim with $W(n)=\mathrm{R}\Gamma_{\fm}(U_n)$ and $v_n = h^0(X,\calf_n)$, to obtain
\[
\lim_{n\to \infty} \frac{h_{i,t}(\bsa; U_n)}{h^0(X,\calf_n)} = \lim_{n\to \infty} \frac{h_{i+d,t}(\bsa; \lch d{\fm}{U_n})}{h^0(X,\calf_n)} = 0
\quad\text{for $ i\ge 1$ and all $t$.}
\]
By \eqref{eq:koszul-bound} the Koszul homology modules are zero outside the range $t_0 <  t < t_1 + |\bsa|$, so the preceding computation yields
\[
\lim_{n\to \infty} \frac{h_{i}(\bsa; U_n)}{h^0(X,\calf_n)} = 0 \qquad\text{for $ i\ge 1$.}
\]
Combining this with \eqref{eq:nu-estimate} we deduce
\[
\lim_{n\to \infty} \frac{h_{i}(\bsa; U_n)}{\nu_A(U_n)} =0 \qquad\text{for $ i\ge 1$.}
\]
Thus the sequence $(U_n)_{n\geqslant 0}$ is lim Cohen-Macaulay as claimed.

To verify that the sequence is lim Ulrich, it remains to verify \eqref{eq:lus}.  Applying \eqref{eq:e-from-hilb} to each $U_n$ and applying property \ref{ch:limUsheaves}(4) yields
\[
\lim_{n\to \infty} \frac{e_d(U_n)}{h^0(X,\calf_n)} = \sum_{0\leqslant i,j\leqslant c} \lim_{n\to \infty} (-1)^{i+j} {\binom cj} \frac{h^i(X, \calf_n(-j))}{h^0(X, \calf_n)} =1\,.
\]
Recalling \eqref{eq:nu-estimate} we obtain
\[
\limsup_{n\to \infty} \frac{e_d(U_n)}{\nu_A(U_n)} \le 1\,.
\]
Combining this inequality with \eqref{eq:lcm-inequality} we deduce that the limit above equals $1$, as desired.
This completes the proof.
\end{proof}

\begin{chunk}
\label{ch:lim-CM-sheaves}
In the preceding proof we first verified that the sequence $(\varGamma_*(\calf_n))_{n\geqslant 0}$ is lim Cohen-Macaulay using only conditions (1)--(3) in \ref{ch:limUsheaves} and the following part of condition (4):  For $1\le i\le c-1$ and each $t$, and for $i=0$ and $t=-1$, one has
\[
\lim_{n\to\infty}\frac{h^i(X,\calf_n(t))}{h^0(X,\calf_n)}=0\,.
\]
One might thus call such a sequence $(\calf_n)_{n\geqslant 0}$ a lim Cohen-Macaulay sequence.
\end{chunk}

\section{Graded rings in positive characteristic}
\label{se:Ma}
In this section we use Theorem~\ref{th:usheaf-umodule} to recover and extend  results of the second author~\cite[Theorem~3.5]{Ma:2020} on the existence of lim Ulrich sequences over standard graded rings defined over a field of positive characteristic.

\begin{theorem}
\label{th:ma-convergence}
Let $k$ be a field of positive characteristic, infinite and perfect, and $R$ the localization of a standard graded $k$-algebra at its homogenous maximal ideal.
There exists a lim Ulrich sequence of $R$-modules $(U_n)_{n\geqslant 0}$ such that
\[
\lim_{n\to\infty} \frac{[U_n]}{\nu_R(U_n)} = \frac{[R]_{d}}{e(R)}
\]
in $\euler R$, for $d\colonequals \dim R$.
\end{theorem}

The proof of this result is given towards the end of this section.  One of the key steps in it is the construction of a lim Ulrich sequence of sheaves over the projective variety defined by the graded ring in question; see Theorem~\ref{th:ma-revisited}. For now we record the following consequence for Dutta multiplicities, recalled in \ref{ch:Dutta}.

\begin{corollary}
\label{co:chi-infinity}
Let $k$ be a field of positive characteristic and $R$ the localization of a standard graded $k$-algebra at its homogenous maximal ideal. Any short complex $F$ in $\perf R$  satisfies
\[
\binom {\dim R}i \chi_{\infty}(F) \ge \beta^R_i(F) e(R)\qquad\text{for each $i$.}
\]
\end{corollary}

\begin{proof}
We can assume $k$ is algebraically closed; see \ref{ch:gonflament} and \ref{ch:Dutta}.  Then $[R]_d$ is a lim Ulrich point in $\euler R$, by Theorem~\ref{th:ma-convergence}, so Corollary~\ref{co:lus} applies.
\end{proof}

\begin{chunk}
\label{ch:chi-graded}
Let $k$ be as in Theorem~\ref{th:ma-convergence} and $A$ a standard graded $k$-algebra. If $F$ is a short complex of finite free graded $A$-modules, with $\hh F$ of finite length, then
\[
\chi(F) = \chi_{\infty}(F) = \chi_{\infty}(F_\fm)
\]
where the first equality can be proved along the lines of that of  \cite[Theorem~2.3]{McDonnell:2011}. The second equality is a simple verification. Thus Corollary~\ref{co:chi-infinity} yields lower bounds on $\chi(F)$.
\end{chunk}

\subsection*{Cohomology on $(\bbP_k^1)^c$}
Let $k$ be a field (with no further restrictions), $c$ a positive integer, and consider the $c$-fold product of the projective line:
\[
Z\colonequals (\bbP_k^1)^c\,.
\]
We will need to compute the cohomology of certain sheaves on $Z$. To that end we recollect some basic results in this topic;  for further details see, for example, \cite{Cox:1995}.

Let $k[\bsx,\bsy] \colonequals k[x_1,y_1;\dots; x_c,y_c]$ be the $\bbZ^c$-graded ring where  $x_i$ and $y_i$ are of multidegree $(0,\dots,1,\dots 0)$, where $1$ is in the $i$th spot. Each multigraded $k[\bsx,\bsy]$-module determines a coherent sheaf on $Z$. In particular for any $c$-tuple $\bsa\colonequals (a_1,\dots,a_c)$ there is a coherent sheaf $\calo_Z(\bsa)$ on $Z$ associated to the $k[\bsx,\bsy]$-module $k[\bsx,\bsy](\bsa)$. The functor assigning finitely generated multigraded modules to coherent sheaves on $Z$ is exact, and every coherent sheaf arises via this construction.  Each finitely generated multigraded $k[\bsx,\bsy]$-module admits a resolution of length at most $2c$ by direct sums of twists of $k[\bsx,\bsy]$. Thus for each coherent sheaf $\caln$ there is an exact sequence of the form
\[
0\lra \calf_{2c}\lra \cdots \lra \calf_0\lra \caln \lra 0\,,
\]
where each $\calf_i$ is a direct sum of copies of $\calo_Z(\bsa)$ for various choices of $c$-tuples $\bsa$. The cohomology of $\calo_Z(\bsa)$ is determined by the isomorphism
\begin{equation}
\label{eq:Z-cohomology}
\CH *Z{\calo_Z(\bsa)} \cong \CH *{\bbP_k^1}{\calo_{\bbP_k^1}(a_1)}\otimes_k\cdots\otimes_k  \CH *{\bbP_k^1}{\calo_{\bbP^1}(a_c)}
\end{equation}
of graded $k$-vector spaces and the computation
\[
h^i(\bbP_k^1, \calo_{\bbP_k^1}(a)) =
\begin{cases}
a+1 & \text{if $i=0$ and $a\ge 0$}\\
-a-1 & \text{if $i=1$ and $a\le -2$} \\
0 & \text{otherwise.}
\end{cases}
\]
Therefore for any integer $a$ one gets
\begin{equation}
\label{eq:Z-cohomology2}
h^*(\bbP_k^1, \calo_{\bbP^1}(a)) = |a+1|\,.
\end{equation}
These computations lead in the usual way to the following analog of the Serre vanishing theorem for projective spaces. In what follows given $\bsa\colonequals (a_1,\dots,a_c)$ and $\bsb=(b_1,\dots,b_c)$ in $\bbZ^c$ we write $\bsb\ge \bsa$ to indicate that $b_i\ge a_i$ for each $i$ and write $\bsa\gg 0$ to indicate that each $a_i\gg 0$ for each $i$.

\begin{lemma}
\label{le:Z-SerreVanishing}
For any coherent sheaf $\caln$ on $Z$ one has
\[
\CH iZ{\caln(\bsa)} =0 \quad\text{for $\bsa\gg 0$ and $i\ge 1$.}
\]
\end{lemma}

\begin{proof}
It follows from the computations above that the desired result holds for sheaves of the form $\calo(\bsa)$. The general case follows since any coherent sheaf  $\caln$ has a finite resolution by direct sum of sheaves of this form; see Lemma~\ref{le:lim-zero}.
\end{proof}

\begin{chunk}
\label{ch:Z-hilb}
The function $\hilb_{\caln}$ defined by
\[
\hilb_{\caln}(\bsa) \colonequals \sum_i (-1)^i h^i(Z, \caln(\bsa)) \qquad\text{for $\bsa\in\bbZ^c$,}
\]
is a polynomial in $c$-variables, called the \emph{Hilbert polynomial} of $\caln$. It has the form
\[
\hilb_{\caln}(z_1,\dots,z_c) = \rank(\caln) z_1\cdots z_c + \text{lower degree terms}.
\]
This can be deduced from the fact that  $\caln$ has a finite resolution consisting of twists of $\calo$ and the following computation
\[
\hilb_{\calo(\bsa)}(z_1,\dots,z_c) = \prod_i (z_i+a_i+1)\,,
\]
which is immediate from \eqref{eq:Z-cohomology}. Moreover, by Serre vanishing one has
\[
\hilb_{\caln}(\bsa) = h^0(Z,\caln(\bsa))  \qquad\text{for $\bsa\gg 0$.}
\]
\end{chunk}

Given an integer $t$ and a positive integer $p$  set
\begin{equation}
\label{eq:bnt}
\bsb_n(t) \colonequals ((t+1)p^n,\dots, (t+c)p^n) \qquad \text{for $n\ge 0$}\,.
\end{equation}
In what follows the asymptotic behavior $\bsb_n(t)$ with respect to $t$ and $n$ will be important. Observe that given any $\bsa\in\bbZ$, there exists an integer $t_0\le 0$ such that $\bsb_n(t)\le \bsa$ for all $t\le t_0$ and all $n$. In the same vein, there exists an integer $t_1\ge 1$ such that $\bsb_n(t)\ge \bsa$ for all $t\ge t_1$ and all $n$. Moreover $\bsb_n(t)\ge \bsa$ for all $t\ge 0$ and $n\gg 0$ and $\bsb_n(t)\le \bsa$ for all $t\le -c-1$ and $n\gg 0$.

The next two results are intended for use in the proof of Theorem~\ref{th:ma-revisited}, which may explain their odd appearance.
Figure~\ref{fig:proof} is a visual guide to the statements. Parts (1), (2), and (3) correspond to regions A, B, and C, respectively.

\begin{figure}[htbp]
\caption[top]{Visual guide to Lemmas~\ref{le:oddZ1} and \ref{le:oddZ2}.}
\label{fig:proof}
\vspace*{-2cm}
\includegraphics[scale=0.95]{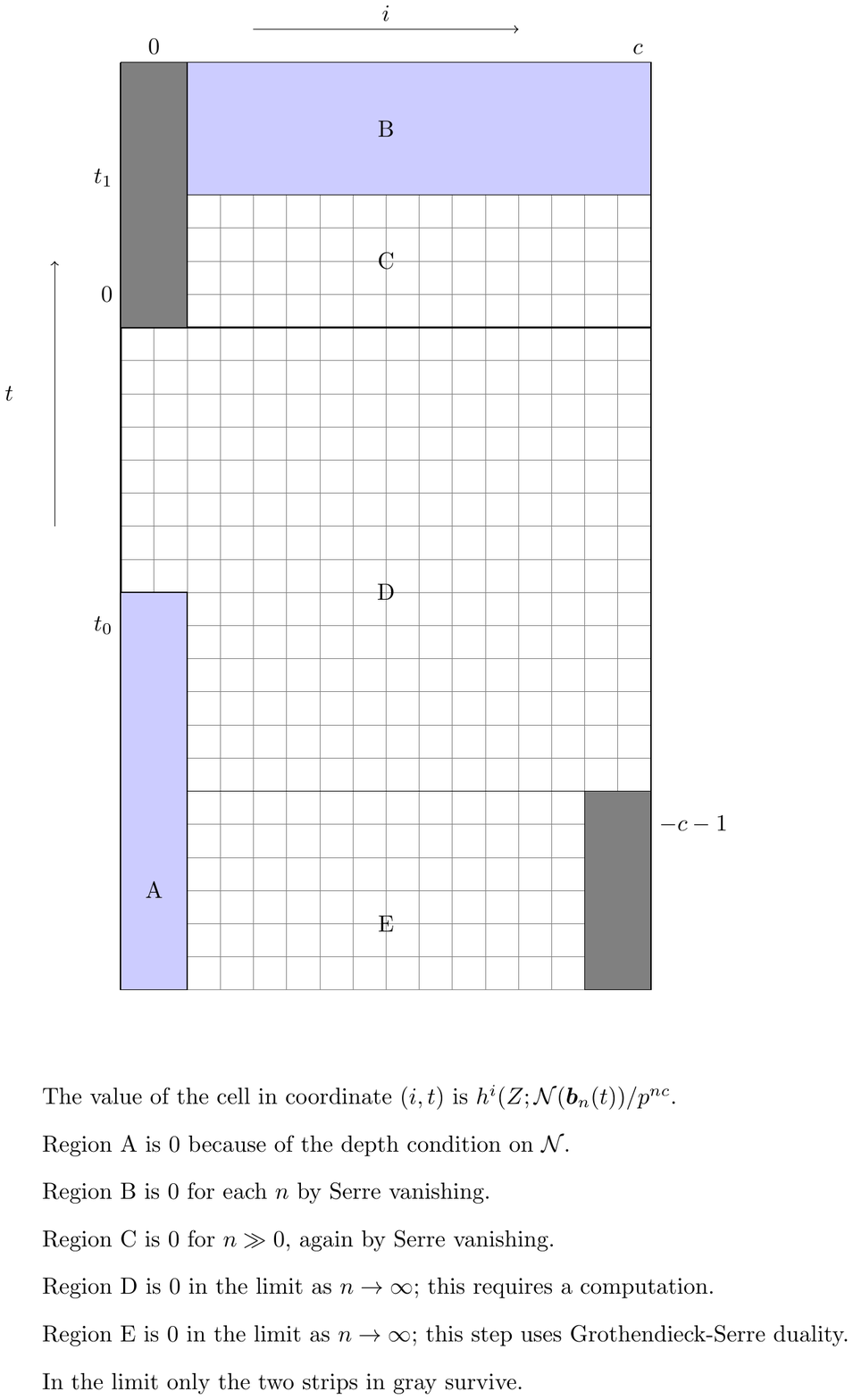}
\label{fig:proof2}
\end{figure}

\begin{lemma}
\label{le:oddZ1}
Let $\caln$ be a coherent sheaf on $Z$ with $\CH 0Z{\caln(\bsa)}=0$ for $\bsa\ll 0$.
\begin{enumerate}[\quad\rm(1)]
\item
There exists an integer $t_0$ such that for $\bsb_n(t)$ as in \eqref{eq:bnt} one has
\[
\CH 0Z{\caln(\bsb_n(t))}=0\quad\text{for $t\le t_0$ and  $n\ge 0$.}
\]
\item
There exists an integer $t_1$ such that
\[
\CH iZ{\caln(\bsb_n(t))}=0 \quad\text{for $t\ge t_1$, $i\ge 1$, and  $n\ge 0$.}
\]
\item
One has
\[
\CH iZ{\caln(\bsb_n(t))}=0 \quad\text{for $t\ge 0$, $i\ge 1$, and  $n\gg 0$.}
\]
\end{enumerate}
\end{lemma}

\begin{proof}
(1)  By hypothesis, there exists $\bsa\in\bbZ^c$ such that  $\CH 0Z{\caln(\bsa'))}=0$ for all $\bsa'\le \bsa$. Choose  $t_0$ such that $\bsb_n(t)\le \bsa$ for all $t\le t_0$ and all $n\ge 0$, so that
\[
\CH 0Z{\caln(\bsb_n(t))} = 0\,.
\]
This justifies (1).

(2) By Serre vanishing, Lemma~\ref{le:Z-SerreVanishing}, there exists $\bsb\in\bbZ^c$ such that $\CH iZ{\caln(\bsb')}=0$ for $\bsb'\ge \bsb$ and $i\ge 1$.  Choose $t_1$ such that $\bsb_n(t)\ge \bsb$ for all $t\ge t_1$ and all $n$. Then
\[
\CH iZ{\caln(\bsb_n(t))} = 0
\]
for $t\ge t_1$, $i\ge 1$, and all $n\ge 0$.  This justifies (2).

(3)  One has  $\bsb_n(t)\ge \bsb'$ for all $t\ge 0$ and $n\gg 0$, where $\bsb'$ is as in the previous paragraph, and hence (3) holds.
\end{proof}

Parts (1) and (2) in the statement below correspond to regions D and E, respectively, in Figure~\ref{fig:proof}. Observe that the conclusions below concern limits; that these limits exist is part of the assertion.

\begin{lemma}
\label{le:oddZ2}
Let $\caln$ be a coherent sheaf on $Z$ with $\CH 0Z{\caln(\bsa)}=0$ for $\bsa\ll 0$.
\begin{enumerate}[\quad\rm(1)]
\item
For all integers $t$ and  $\bsb_n(t)$ as in \eqref{eq:bnt} one has
\[
\lim_{n\to\infty} \frac{h^*(Z,\caln(\bsb_n(t)))}{p^{nc}} = \rank(\caln) |(t+1)\cdots (t+c)|
\]
In particular the limit is zero when $-c\le t\le -1$.
\item
When  $t\le -c-1$ and $i\le c-1$ one has
\[
\lim_{n\to\infty} \frac{h^i(Z,\caln(\bsb_n(t)))}{p^{nc}}=0\,.
\]
\end{enumerate}
\end{lemma}

\begin{proof}
We begin by proving (1) when $t\ge 0$; this corresponds to the region C in Figure~\ref{fig:proof}.  In this case  it follows from Lemma~\ref{le:oddZ1}(3) that for $n\gg 0$ one has
\[
h^*(Z;\caln(\bsb_n(t))) = h^0(Z;\caln(\bsb_n(t))) = \hilb_{\caln}(\bsb_n(t))\,.
\]
Given the description of $\bsb_n(t)$ from \eqref{eq:bnt} it follows that for $n\gg 0$ one has
\[
\hilb_{\caln}(\bsb_n(t))  = \rank(\caln) (t+1)\cdots(t+c)p^{nc} + (\text{terms with $p^s$ for $s<nc$.})
\]
Combining the two observations above yields
\[
\lim_{n\to\infty} \frac{h^*(Z,\caln(\bsb_n(t)))}{p^{nc}} = \rank(\caln) (t+1)\cdots(t+c)\,.
\]
This is the desired result.

Suppose $-c\le t\le -1$; this corresponds to the region D in Figure~\ref{fig:proof}. For any $\bsa\in \bbZ^c$, from \eqref{eq:Z-cohomology} and \eqref{eq:Z-cohomology2} one gets
\[
h^*(Z, \calo(\bsa)(\bsb_n(t))) = h^*(Z,\calo(\bsb_n(t)+\bsa)) = \prod_{i=1}^c |(t+i)p^n+a_i+1|\,.
\]
The constraint on $t$ means that the coefficient of $p^n$ in one of the terms in the product is zero, so the highest power of $p$ that appears is $p^{(n-1)c}$. It follows that
\[
\lim_{n\to\infty} \frac{h^*(Z, \calo(\bsa)(\bsb_n(t)))}{p^{nc}}=0\,.
\]
Since $\caln$ has a finite resolution by direct sum of sheaves of the form $\calo(\bsa)$ for $\bsa$ the stated vanishing holds also for $\caln$.

We verify (1) for $t\le -c-1$ and (2) simultaneously. We are now discussing region E in Figure~\ref{fig:proof}. The canonical sheaf on $Z$ is $\calo(-2,\dots,-2)$, so setting
\[
\bsd_n(t) = -\bsb_n(t) - (2,\dots,2)
\]
and applying Grothendieck-Serre duality we get
\[
{\CH iZ{\caln(\bsb_n(t))}}^{\vee} \cong  \CH {c-i}Z{\mathrm{RHom}_{\calo_Z}(\caln,\calo_Z)(\bsd_n(t))}
\]
The cohomology sheaves of $\mathrm{RHom}_Z(\caln,\calo_Z)$ are the coherent sheaves
\[
\cale^j \colonequals {\mathcal{E}{xt}}^j_{Z}(\caln,\calo_Z) \quad\text{for $0\le j\le c$.}
\]
Observe that since $t\le -c-1$ the $c$-tuple $\bsd_n(t)\gg 0$ for $n\gg 0$, so arguing as for (1) in the case $t\ge 0$ we get
\[
\lim_{n\to\infty} \frac{h^*(Z,\cale^j(\bsd_n(t)))}{p^{nc}}
	= \lim_{n\to\infty} \frac{\hilb_{\cale^j}(\bsd_n(t)))}{p^{nc}}\,.
\]
From the description of $\bsd_n(t)$,  for $n\gg 0$ we get that
\[
\hilb_{\cale^j}(\bsd_n(t)) = \rank(\cale^j) (-t-1)\cdots(-t-c)p^{nc} + (\text{terms with $p^s$ for $s<nc$.})
\]
The preceding computations yield
\[
\lim_{n\to\infty} \frac{h^*(Z,\cale^j(\bsd_n(t)))}{p^{nc}}  = \rank(\cale^j) (-t-1)\cdots(-t-c)\,.
\]
When $j\ge 1$ the sheaf $\cale^j$ has rank $0$ and the limit above is zero. Given this it is easy to see---see also Lemma~\ref{le:lim-zero}---that
\[
\lim_{n\to\infty} \frac{h^{c-i}(Z, \mathrm{RHom}_Z(\caln,\calo_Z) (\bsd_n(t)))}{p^{nc}} = \lim_{n\to\infty} \frac{h^{c-i}(Z,\cale^0(\bsd_n(t)))}{p^{nc}}\,.
\]
In particular when $i\le c-1$  the terms in the limit on the righthand side are $0$ for $n\gg 0$, again by Serre vanishing.  From this observation and the Grothendieck-Serre duality isomorphism above one gets
\[
\lim_{n\to\infty} \frac{h^i(Z,\caln(\bsb_n(t)))}{p^{nc}} = 0 \qquad\text{for $i\le c-1$.}
\]
This settles (2). Moreover  the same tokens give
\begin{align*}
\lim_{n\to\infty} \frac{h^*(Z,\caln(\bsb_n(t)))}{p^{nc}}
	&= \lim_{n\to\infty} \frac{\hilb_{\cale^0}(\bsd_n(t))}{p^{nc}}\\
	&= \rank(\cale^0) (-t-1)\cdots(-t-c)
\end{align*}
It remains to observe that $\cale^0=\mathcal{H}{om}(\caln,\calo_Z)$ so its rank equals $\rank\caln$.
\end{proof}

In the rest of this section we return to the context of Theorem~\ref{th:ma-convergence}.

\begin{chunk}
\label{ch:PandZ}
Let $k$ be an infinite field,  perfect of positive characteristic $p$, and  let $A$ be a standard graded $k$-algebra. We will assume in addition that $d\colonequals \dim A\ge 2$. This puts us in the context of Section~\ref{se:lim-Ulrich-sheaves} and we adopt  the notation from there:
\[
X\colonequals \proj A \qquad\text{and}\qquad c\colonequals \dim X = d-1
\]
We write $\vf\colon X\to X$ for the Frobenius map. Since $k$ is perfect this is a finite morphism. We fix maps
\begin{equation}
\label{eq:PandZ}
X \xra{\ \pi \ } \bbP_k^c \xleftarrow{\ \rho\ }  Z\colonequals (\bbP^1_k)^c\,,
\end{equation}
where $\pi$ is a finite linear projection, which exists because $k$ is infinite, and $\rho$ is a finite flat map such that $\rho^*(\calo_{\bbP_k^c}(1)) = \calo_Z(1,\dots,1)$.  For example, $\rho$ could be the  map given quotienting by the action of the symmetric group $S_c$ on $Z$.  For each non-negative integer $n$ consider the line bundle on $Z$ given by
\[
\call_n \colonequals \calo_Z(p^n, 2p^n,\dots, cp^n)
\]
\end{chunk}

We will also need a coherent sheaf on $X$ with special properties.

\begin{chunk}
\label{ch:calm}
Let $\calm$ be a coherent sheaf on $X$ such that the following conditions hold:
\begin{enumerate}[\quad\rm(1)]
\item
$\CH 0X{\calm(t)}=0$ for $t\ll0$;
\item
The sheaf $\pi_*\calm$ on $\bbP^c_k$ has positive rank;
\end{enumerate}
 With $\calm$ such a sheaf, and notation as above, we construct the sequence of sheaves on $X$ specified by
\begin{equation}
\label{eq:calun}
\calu_n(\calm) \colonequals \vf^n_*(\calm \otimes \pi^*\rho_*(\call_n))\,.
\end{equation}
\end{chunk}

The result below is the first step in the proof of Theorem~\ref{th:ma-convergence}.

\begin{theorem}
\label{th:ma-revisited}
With $X$ and $\calm$ as above, the sequence $(\calu_n(\calm))_{n\geqslant 0}$ of sheaves on $X$ given in \eqref{eq:calun} is a lim Ulrich sequence.
\end{theorem}

\begin{proof}
Set $\calu_n\colonequals \calu_n(\calm)$ and $\caln \colonequals \rho^*\pi_*\calm$; this is a sheaf on $Z$. Since $\rho$ is flat the hypotheses on $\calm$ imply that $\caln$ is locally of positive depth at all closed points, and has positive rank; see the discussion in \ref{ch:fac}.

It will be convenient to introduce, as in \eqref{eq:bnt}, the $c$-tuple
\[
\bsb_n(t) \colonequals ((t+1)p^n,\dots, (t+c)p^n)\,,
\]
for $n\ge 0$ and all $t$, as in \eqref{eq:bnt}.

\begin{claim}
For any  integers $n,t$ and $i$ there is an equality
\[
h^i(X,{\calu_n(t)}) = h^i(Z;\caln(\bsb_n(t))) \,.
\]

Indeed,  since $\vf^*(\calo_X(t)) \cong \calo_X(pt)$ the projection formula for $\vf$ yields
\[
\calu_n(t) = \vf^n_*(\calm \otimes \pi^*\rho_*(\call_n))(t)\cong \vf^n_*(\calm \otimes \pi^*\rho_*(\call_n)(p^nt))\,.
\]
Since $k$ is perfect, for any coherent sheaf $\mathcal G$ on $X$ one has $\CH iX{\vf_*\mathcal{G}}\cong \CH iX{\mathcal G}$. Therefore from the isomorphism above one gets the first isomorphism below
\begin{align*}
\CH iX{\calu_n(t)}
	&\cong \CH iX{\calm\otimes (\pi^*\rho_*\call_n)(p^nt)}\\
	&\cong \CH i{\bbP_k^c}{\pi_*(\calm)\otimes \rho_*\call_n(p^nt)}\,.
\end{align*}
The second isomorphism holds by the the projection formula for $\pi$ and its linearity. The stated claim now follows from the projection formula for $\rho$ and the fact that, by construction, $\rho$ is flat, with $\rho^*\calo_{\bbP_k^c}(1)\cong \calo_Z(1,\dots,1)$
\end{claim}

Given the  claim above and the computations on $Z$ made earlier, it is easy to check that $(\calu_n)_{n\geqslant 0}$ is a lim Ulrich sequence of sheaves. Indeed, Lemma~\ref{le:oddZ1}(3) and Lemma~\ref{le:oddZ2}(1) imply
\[
\lim_{n\to\infty} \frac{h^0(X, \calu_n(0))}{p^{nc}}
	= \lim_{n\to\infty} \frac{h^*(X, \calu_n(0))}{p^{nc}}
	= (\rank\caln) c!\,.
\]
Since $\rank\caln >0$ condition \ref{ch:limUsheaves}(1) follows.

Conditions \ref{ch:limUsheaves}(2) and \ref{ch:limUsheaves}(3) follow from parts (1) and (2) of Lemma~\ref{le:oddZ1}, respectively. Lastly, to verify  \ref{ch:limUsheaves}(4) for the sequence $(\calu_n)_{n\geqslant 0}$, we verify that
\[
\lim_{n\to\infty} \frac{h^i(Z;\caln(\bsb_n(t)))}{p^{nc}}=0
\]
for all $i,t$, except when $i=0$ and $t\ge 0$, and when $i=c$ and $t\le -c-1$.  When $t\ge 0$ this is by Lemma~\ref{le:oddZ1}(3), whereas when $t\le -1$ the desired vanishing is a consequence of Lemma~\ref{le:oddZ2}. This completes the proof.
\end{proof}

Next we describe a map that mediates the passage between coherent sheaves on $X$ and $R$-modules.

\begin{chunk}
\label{ch:gamma}
Let $A$ and $X$ be as in  \ref{ch:PandZ} and $R\colonequals A_\fm$, the localization of $A$ at its homogeneous maximal ideal. The functor 
\[
\calf\mapsto \varGamma_{n\geqslant 0}(\calf)_\fm
\]
from coherent sheaves on $X$ to finitely generated $R$-modules yields an $\bbR$-linear map
\[
\gamma \colon {\grog X}_{\bbR}  \lra  {\grog R}_{\bbR}\,.
\]
Observe that when the $A$-module $\Gamma_*(\calf)$ is finitely generated, its class of $\Gamma_*(\calf)$ in ${\grog A}_{\bbR}$ is the same as that of $\varGamma_{n\geqslant 0}(\calf)$, and hence $\gamma(\calf)=\varGamma_*(\calf)_\fm$. This is because the class of $k$ in $\grog A$ is torsion, given our assumption that $\dim A\ge2$.

We write $\vf_*$  for the pushforward of the Frobenius map on $X$ and also for that on $A$ and $R$; the one in use will be clear from the context. The result below records the compatibility of  $\gamma$ with $\vf_*$, and with cap product, $\cap$, which is induced by the tensor product of coherent sheaves. 

\begin{lemma}
\label{le:gamma}
The following statements hold.
\begin{enumerate}[\quad\rm(1)]
\item
There is an equality $\vf_*\circ \gamma = p\cdot \gamma\circ \vf_* $ as maps from ${\grog X}_{\bbR}$ to ${\grog R}_{\bbR}$.
\item
With $\pi\colon X\to \bbP^c_k$ the map from \eqref{eq:PandZ}, there is an equality
\[
\gamma(\alpha\cap \pi^*(\beta)) = \gamma(\alpha)\cdot \rank\beta\,,
\]
for any $\alpha$ in ${\grog X}_{\bbR}$ and $\beta$ in ${\mathrm{K}_0(\bbP_k^c)}_{\bbR}$.
\end{enumerate}
\end{lemma}

\begin{proof} 
The arguments uses the  fact  for any coherent sheaf $\calf$ on $X$ the class  $[\calf] - [\calf(-1)]$ in ${\grog X}$ is in the kernel of the map $\gamma$.

 (1) Fix a coherent sheaf $\calf$ on $X$. If  $M$ is a finitely generated graded $A$-module whose sheafification is $\calf$, then $\vf_*(\calf(i))$ is the sheafication of the abelian group $\oplus_{j\in \bbZ} M_{i+jp}$, viewed as an $A$-module via the Frobenius endomorphism of $A$. Given this observation, in $\grog{R}_{\bbR}$ one gets
\[
\vf_*  \gamma (\calf) = \sum_{i=0}^{p-1} \gamma  \vf_*(\calf(i))  
\]
To justify (1) it remains to verify that for any $\alpha$ in ${\grog X}_{\bbQ}$ and integer  $i$ one has
\[
\gamma\vf_*(\alpha(i)) = \gamma\vf_*(\alpha)\,.
\]
In verifying this equality we may assume $i = 1$.  Moreover the element
\[
 1 + [\calo(1)] + \dots + [\calo(p-1)]
 \]
is a unit in ${\mathrm{K}_0(X)}_{\bbR}$, by \cite[Chapter II, Proposition~8.8.4]{Weibel:2013},  and ${\grog X}_{\bbR}$ is a ${\mathrm{K}_0(X)}_{\bbR}$-module. Therefore we may assume
\[
\alpha = (1 + [\calo(1)] + \dots + [\calo(p-1)]) \beta
\]
for some class $\beta$. Hence
\[
\vf_*(\alpha - \alpha(1)) = \vf_*((1 - [\calo(p)]) \beta) = \vf_*(\beta) - \vf_*(\beta(p)) = \vf_*(\beta) - \vf_*(\beta)(1)
\]
where the last equality uses the projection formula.  Now apply $\gamma$ to conclude
\[
\gamma(\vf_*(\alpha - \alpha(1)) = 0 
\]
and hence that $\gamma(\vf_*(\alpha)) = \gamma(\vf_*(\alpha(1))$, as desired.

(2) This holds because ${\mathrm{K}_0(\bbP^c)}_{\bbR}$ is generated by the classes of twists of $\calo_{\bbP_k^c}$. 
\end{proof}

\end{chunk}

We are now ready to prove the result announced at the start of this section.

\begin{proof}[Proof of Theorem~\ref{th:ma-convergence}]
Let $A$ be a standard graded $k$-algebra as postulated in the statement, with homogenous maximal ideal $\fm$, so $R=A_\fm$. With $d\colonequals \dim A\le 1$ the desired result is contained in Example~\ref{ex:dim1U}. So we assume $d\ge 2$.

We keep the notation from \ref{ch:PandZ}; in particular,  $X\colonequals \proj A$.  The first step is to get a lim Ulrich sequence of $R$-modules by applying Theorem~\ref{th:ma-revisited}. This result takes as  input a coherent sheaf $\calm$ on $X$ satisfying the conditions in \ref{ch:calm}.  For what follows, we need a bit more from $\calm$: Since $A$ is graded, its associated primes are homogenous. Consider the set $\Lambda(A)$ of minimal primes of $A$ defining components of dimension $d$. Let $\calm$ be the sheaf on $X$ obtained by sheafifying the $A$-module
\[
\bigoplus_{\fp\in\Lambda(A)} (A/\fp)^{\length(A_\fp)}\,.
\]
Since the module is finitely generated $\calm$ is coherent.

\begin{claim}
The sheaf $\calm$  satisfies the conditions in \ref{ch:calm},  the $A$-module $M\colonequals \varGamma_*(\calm)$ is finitely generated, and the $R$-module $M_\fm$ satisfies condition \eqref{eq:rank}.
\medskip

Indeed, as $\dim (A/\fp)\ge 2$ for  $\fp\in\Lambda(A)$, the sheaf $\calm$ is locally of positive depth at each closed point, so condition (1) of \ref{ch:calm} holds; equivalently,  the $A$-module $M$ is finitely generated; see \ref{ch:fac}. Moreover
\[
\dim \pi_*\calm = \dim \calm = \dim X = \dim \bbP^c_k
\]
where the first and last equalities hold because $\pi$ is finite, and the second holds by the choice of $M$. Thus $\pi_*\calm$ has positive rank, so satisfying condition (2) of \ref{ch:calm}.

Finally, by the construction of $M$ there is a natural map of $A$-modules
 \[
 \bigoplus_{\fp\in\Lambda(A)} (A/\fp)^{\length(A_\fp)}\lra M \,.
 \]
This map is one-to-one, because the module on the left has positive depth, and its cokernel has finite length. Now it is a simple computation to check that the $R$-module $M_\fm$ satisfies \eqref{eq:rank}.
\end{claim}

For each $n\ge 0$, let $\calu_n$ be the sheaf on $X$ constructed as in \eqref{eq:calun} with $\calm$ as defined above and consider the $R$-module:
\[
U_n\colonequals \varGamma_*(\calu_n)_\fm\,.
\]
 Theorem~\ref{th:ma-revisited} yields that  $(\calu_n)_{n\geqslant 0}$ is a lim Ulrich sequence of sheaves on $X$, and then Theorem~\ref{th:usheaf-umodule} yields that the sequence of $A$-modules $(\varGamma_*(U_n))_{n\geqslant 0}$ is lim Ulrich. It remains to invoke Proposition~\ref{pr:graded-vs-local}(3) to conclude that the sequence of $R$-modules $(U_n)_{n\geqslant 0}$ is lim Ulrich.

Next we verify the following  equality:
\[
\lim_{n\to\infty} \frac{[U_n]}{\nu_R(U_n)} = \frac{[R]_{d}}{e(R)}\,.
\]
We do so by passing to $X$ and using the observations from \ref{ch:gamma}. 

Recall the construction of the sheave $\calu_n$ from \eqref{eq:calun}. The rank of $\rho_*(\call_n)$ is $c!$ and hence so is the rank of $\pi^*\rho_*(\call_n)$. Given this and using Lemma~\ref{le:gamma} one gets:
\begin{align*}
p^n[U_n] & = p^n\gamma [\calu_n]\\
	&= p^n \gamma [\vf_*^n(\calm \otimes \pi^*\rho_*(\call_n))] \\
	& =  \vf_*^n \gamma [\calm \otimes \pi^*\rho_*(\call_n))] \\
	& = c!  \vf_*^n (\gamma [\calm])\\
	&= c! [\vf_*^n(M_\fm)]
\end{align*}
From this equality one gets the second equality below:
\begin{align*}
\lim_{n\to\infty} \frac{[U_n]}{\nu_R(U_n)}
	&=\lim_{n\to\infty} \frac{[U_n]}{e_d(U_n)} \\
	&= \lim_{n\to\infty} \frac{[\vf^n_*(M_\fm)]}{e_d(\vf^n_*(M_\fm))} \\
	&= \frac{[R]_d}{e(R)}
\end{align*}
The first equality is  by \eqref{eq:lus}, for the sequence $(U_n)$ is lim Ulrich. The last one is by Lemma~\ref{le:frobenius-limit}, which can applied given the claim from earlier in this proof.
\end{proof}

\section{Lech's conjecture revisited}
\label{se:lech}
 Corollary~\ref{co:chi-infinity}  leads us to propose the following:

\begin{conjecture}
\label{con:DuttaMultiplicity}
Let $R$ be a complete local ring. If $F$ is a short complex in $\perf R$, then  $\chi_\infty(F)\geq e(R)$.
\end{conjecture}

We are no longer assuming $R$ contains a field of positive characteristic, so we have to recall what Dutta multiplicity is in this generality: The ring $R$, being complete, is a quotient of a regular local ring, so one has the Riemann-Roch isomorphism
\[
\tau\colon {\grog R}_\bbQ \xra{\ \cong\ } A_*(R)_{\bbQ} = \bigoplus_{i=0}^{d} A_i(R)_\bbQ
\]
where the target is the rational Chow group of $R$. Here $d=\dim R$ as usual. The map $\tau$ is hard to describe in general, except for the component in degree $d$, for  $A_d(R)_\bbQ$ has a basis consisting of classes $[V(R/\fp)]$, for $\fp\in \Lambda(R)$ as in \eqref{eq:g0}, and
\[
\tau_d \colon {\grog R}_{\bbQ} \lra \bbR^{\#\Lambda(R)} \quad\text{is induced by $M\mapsto (\length_{R_\fp}M_\fp)_{\fp\in\Lambda(R)}$.}
\]
Compare with Lemma~\ref{le:Feigen}. Writing $[R]_d$ for inverse image of $\tau_d(R)$, the Dutta multiplicity of $F$ in $\perf R$ is
\[
\chi_{\infty}(F) \colonequals \pair F{[R]_d}
\]
where $\pair --$ is the usual pairing; see \ref{ch:pairing}.  Usually, the Dutta multiplicity is defined using the action of $\knot R$ on the Chow group but the form above is better suited for our approach;  see \cites{Kurano/Roberts:2000, Roberts:1989, Piepmeyer/Walker:2009} for details.

We are going to prove that Conjecture~\ref{con:DuttaMultiplicity} implies Lech's  conjecture~\cite{Lech:1960}:

\begin{chunk}
\label{ch:lech}
Lech's conjecture: Given a flat local map $R\to S$ one has $e(R)\le e(S)$.
\end{chunk}

We refer to \cite{Ma:2017, Ma:2020} for a discussion on the history and current state of this conjecture, recording only that in \cite{Ma:2020}  it is verified when $R$ is standard graded. A key step in the proof is the construction of lim Ulrich sequences. Later in this section we exhibit rings that  have no lim Ulrich sequences, but for which Conjecture~\ref{con:DuttaMultiplicity} holds. The result below suggests an alternative approach to Lech's conjecture.

\begin{proposition}
\label{pr:lech}
If Conjecture \ref{con:DuttaMultiplicity} holds for all complete local rings, then Lech's conjecture holds.
\end{proposition}

\begin{proof}
Let $R\to S$ be a flat local map. A standard reduction allows one to assume $R$ is a complete local domain, $S$ is complete with algebraically closed residue field, and $\dim(S)=\dim(R)$; see \cite[Lemma 2.2]{Ma:2017}. Then one has a Cohen factorization:
\[
R \lra T \twoheadrightarrow S =T/J\,.
\]
where $R\to T$ is flat with regular closed fiber; see \cite[Theorem~1.1]{Avramov/Foxby/Herzog:1994}. In particular,  $e(R)=e(T)$ and the ideal $J$ of $T$ is perfect; see \cite[Lemma 3.5]{Ma:2017}. Moreover  all the fibers of $R\to T$ are complete intersections; see Taba\^a \cite[Th\'eor\`eme 2]{Tabaa:1984}.  In particular, since each minimal prime $\fp$ of $J$ contracts to $0$ in $R$, it follows that the local ring $T_\fp\cong T_\fp/(\fp \cap R)T_\fp$ is a complete intersection for every minimal prime $\fp$ of $J$.

The residue field of $T$ is infinite so we can pick elements $\bsx\colonequals x_1,\dots,x_d$ in $T$ that form part of a system of parameters on $T$ and whose images in $S$ generate a minimal reduction of its maximal ideal, say $\fn$. Let $P$ be a minimal $T$-free resolution of $S$. Since $J$ is perfect the length of $P$ is equal to $h\colonequals \height(J)$.

Set $n\colonequals \dim(T) = d+h$ and $F\colonequals \kos{\bsx} T\otimes_T P$;  observe that this is a short complex in $\mathcal{F}^{\fn}(T)$.  Conjecture \ref{con:DuttaMultiplicity}  yields $\chi_\infty(F)\geq e(T)=e(R)$. Moreover, since $\bsx$ generates a minimal reduction of $\fn$ in $S$, we know that $\chi(F)=e(S)$. Thus to complete the proof it suffices to verify:

\begin{claim}
With notation as above, we have $\chi(F)=\chi_\infty(F)$.
\end{claim}
The assignment $M\mapsto \sum_i (-1)^i [\HH i{P\otimes_TM}]$ induces a map
\[
 P \cap - \colon  \grog T \lra \grog {S}\,.
\]
Consider  the multiplicity 
\[
e_d(-) \colon \grog {S} \lra \bbZ\,;
\]
see Corollary~\ref{co:ed-symbol}.  We claim that there are equalities
\[
\chi(F) = e_d( P\cap [T]) \qquad \text{and}\qquad \chi_{\infty}(F) = e_d(P\cap [T]_n)\,.
\]
Indeed, both  equalities are special cases of the more generality equality
\[
e_d(P\cap - ) = \pair F- \qquad\text{on $\grog T$.}
\]
This holds because for any finitely generated $T$-module $M$ one has equalities
\begin{align*}
\pair FM
	& = \chi(F\otimes_RM) \\
	& = \chi( \kos{\bsx}T \otimes_T P \otimes_TM) \\
	& = \sum_i (-1)^i \chi ( \kos{\bsx}{\HH i{P\otimes_TM}}) \\
	&= \sum_i (-1)^i e_d(\HH i{P\otimes_TM}) \\
	&= e_d(P\cap M)\,.
\end{align*}

Let $\Lambda$ denote the minimal primes of $S$ defining its components of maximal dimension. Since $e_d(-)$ factors through the localization map
\[
\grog {S} \lra \bigoplus_{\fp\in\Lambda} \grog {S_{\fp}}
\]
it suffices to verify that for each minimal prime $\fp$ of $S$ one has
\[
 (P\cap [T])_{\fp} =  (P\cap [T]_n)_\fp\qquad \text{in $\grog {S_{\fp}}$.}
\]
The Riemann-Roch map $\tau$ commutes with localization, in that we have a commutative diagram
\[
\begin{tikzcd}
{\grog T}_{\bbQ}\ar[d] \ar[r, "\tau", "\cong" swap] & {A_*(T)}_{\bbQ}\ar[d] \\
{\grog {T_\fp}}_{\bbQ} \ar[r, "\tau_\fp", "\cong" swap] & {A_*(T_\fp)}_{\bbQ}
\end{tikzcd}
\]
where the map on the right has a degree shift by $h-n$; see \cite[Theorem~18.3 and \S20.1]{Fulton:1998}. Therefore one gets equalities
\[
([T]_n)_\fp = \tau_\fp^{-1} (\tau_n([T])_\fp ) = \tau_\fp^{-1} (\tau_h([T_\fp]) ) = [T_\fp]_h\,.
\]
This leads to following computation where the last one is clear from the definition:
\[
(P\cap [T]_n)_\fp  =  P_\fp \cap [T_\fp]_h  =  P_\fp \cap [T_\fp] =  (P\cap [T])_{\fp}
\]
The second equality holds because $[T_\fp]_h = [T_\fp]$, for $T_\fp$ is a complete intersection ring; see \cite[Proposition~12.4.4(2)]{Roberts:1998}, and also \cite[Corollary~18.1.2]{Fulton:1998}. This completes the proof of the claim and hence of the desired result.
\end{proof}

\subsection*{Non-existence of lim Ulrich sequences}
Recently Yhee~\cite{Yhee:2021} has constructed local rings that do not possess  lim Ulrich sequences. In the following paragraphs we state a version of her results that suits us, referring to \cite{Yhee:2021} for the more general statements and proofs.  Our main goal here is to verify that Conjecture~\ref{con:DuttaMultiplicity} holds for these rings nevertheless.

\begin{chunk}
\label{ch:wlcm}
Let $R$ be a local ring. As in \cite{Ma:2020}, a sequence $(U_n)_{n\geqslant 0}$ of finitely generated $R$-modules is said to be a \emph{weakly lim Cohen-Macaulay sequence} if each $U_n$ is nonzero, and for some system of parameters $\bsr\colonequals r_1,\dots,r_d$ of $R$ one has
\[
\lim_{n\to \infty} \frac {\chi_1(\kos{\bsr}{U_n})}{\nu_R(U_n)}=0,
\]
where $\chi_1(\kos{\bsr}{U_n})\colonequals \sum_{i=1}^d(-1)^{i-1}\length_RH_i(\bsr; U_n)$ is the first Euler characteristic of Koszul complex $\kos{\bsr}{U_n}$.  If in addition \eqref{eq:lus} holds, then  $(U_n)_{n\geqslant 0}$ is said to be a \emph{weakly lim Ulrich sequence}.
\end{chunk}

It is clear from the definitions \ref{ch:lcm}, \ref{ch:lus} and \ref{ch:wlcm} that any lim Cohen-Macaulay sequence is weakly lim Cohen-Macaulay, and that any lim Ulrich sequence is weakly lim Ulrich. In \cite{Ma: 2020}, it is proved that the existence of weakly lim Ulrich sequences implies Lech's conjecture, and it is asked whether every complete local domain of positive characteristic admits a lim Ulrich or weakly lim Ulrich sequence \cite[Question 3.9]{Ma: 2020}. Following \cite{Yhee:2021}, we describe a ring for which no such sequence exists.

\begin{chunk}
\label{ch:Yhee}
Let $k$ be a field and $S\colonequals k\pos{x,y}$, the ring of formal power series over $k$ in indeterminates $x,y$. Then $S$ is a local ring with maximal ideal $\fn\colonequals (x,y)$. Choose an $\fn$-primary ideal $J$ and set
\[
R\colonequals k + J \subseteq S
\]
Then $R$ is a local $k$-subalgebra of $S$, with maximal ideal $\fm\colonequals J$. Moreover $S/R$ has finite length, so $S$ is module-finite over $R$. 
\end{chunk}

The following result follows from \cite[Remark~3.7 and Theorem~5.1]{Yhee:2021}.

\begin{theorem}
\label{th:nolimU}
With notation as in \ref{ch:Yhee}, if $J$ is not generated by a system of parameters of $S$, then $R$ does not have weakly lim Ulrich sequence. \qed
\end{theorem}

Although the ring $R$  defined in \ref{ch:Yhee} has no weakly lim Ulrich sequences, Conjecture~\ref{con:DuttaMultiplicity}  holds for that ring.

\begin{proposition}
\label{pr:Yhee-ring-ok}
Let $R$ be as in \ref{ch:Yhee}. Each short complex $F$ in $\perf R$ satisifes
\[
\chi_{\infty}(F)\ge e(R)\,.
\]
\end{proposition}

\begin{proof}
One has that $e(R) = e(J,S)$, so the desired result is that $\chi_{\infty}(F) \ge e(J,S)$. Since $R$ is a complete local domain of dimension two one has $\chi_{\infty}(F)=\chi(F)$; see \cite[Proposition~3.7]{Kurano:2004}. Moreover as an $R$-module $S/R\cong (x,y)S/J$ and hence the classes of $S$ and $R$ coincide in ${\grog R}_{\bbQ}$. It follows that
\[
\chi_{\infty}(F) = \chi(F) = \chi(S\otimes_RF)\,.
\]
 Set $G\colonequals S\otimes_RF$. Evidently this is a short $S$-complex, and since $S$ is regular, in particular Cohen-Macaulay, $G$ has homology only in degree $0$, that is to say, it is a minimal free resolution of $\HH 0G$. So we have to verify the following: Given  the minimal free resolution of an $S$-module $M$ of finite length:
\[
G\colonequals 0\lra S^c \lra S^b \lra S^a\lra 0
\]
where the differential on $G$ satisfies $d(G)\subseteq JG$, then $\length_SM\ge e(J,S)$.

When $a=1$, so that $M$ is cyclic, we are in a situation where the Hilbert-Burch theorem applies and $G$ has the form
\[
G\colonequals 0\lra S^{b-1} \xra{\ d \ } S^{b} \xra{(s_1  \dots  s_b)} S\lra 0
\]
where $\bss\colonequals s_1,\dots,s_b$ are the minors of size $b-1$ of a matrix representing $d$. When $b=2$ the sequence $s_1,s_2$ is regular, which gives the second equality below:
\[
\length_SM = \length_S(S/(s_1,s_2)) = e((s_1,s_2),S) \ge e(J,S) \,,
\]
The other (in)equalities are clear. This settles the case $b=2$. When $b\ge 3$, as $d(S^{b-1})\subseteq J S^b$ it follows that $(\bss)\subseteq J^{b-1}$. Therefore we get
\[
\length_S M = \length_S (S/(\bss)) \ge \length_S (S/J^{b-1}) \ge \frac 12 e(J^{b-1},S)= \frac{(b-1)^2}{2} e(J,S)
\]
where the second inequality is Lech's inequality~\cite[Theorem~3]{Lech:1960}.   This is the desired estimate for $b\ge 3$.

When $a\ge 2$, one has
\[
\length_SM \ge 2\length_S{(S/J)}\ge e(J,S)
\]
where the second inequality again by Lech's inequality, now applied to $J$.
\end{proof}


\begin{bibdiv}
\begin{biblist}

 \bib{Apassov:1999}{article}{
   author={Apassov, Dmitri},
   title={Annihilating complexes of modules},
   journal={Math. Scand.},
   volume={84},
   date={1999},
   number={1},
   pages={11--22},
   issn={0025-5521},
   review={\MR{1681740}},
   doi={10.7146/math.scand.a-13929},
}


\bib{Avramov/Buchweitz/Iyengar/Miller:2010a}{article}{
   author={Avramov, Luchezar L.},
   author={Buchweitz, Ragnar-Olaf},
   author={Iyengar, Srikanth B.},
   author={Miller, Claudia},
   title={Homology of perfect complexes},
   journal={Adv. Math.},
   volume={223},
   date={2010},
   number={5},
   pages={1731--1781},
   issn={0001-8708},
   review={\MR{2592508}},
   doi={10.1016/j.aim.2009.10.009},
}

\bib{Avramov/Foxby:1998}{article}{
   author={Avramov, Luchezar L.},
   author={Foxby, Hans-Bj\o rn},
   title={Cohen-Macaulay properties of ring homomorphisms},
   journal={Adv. Math.},
   volume={133},
   date={1998},
   number={1},
   pages={54--95},
   issn={0001-8708},
   review={\MR{1492786}},
   doi={10.1006/aima.1997.1684},
}

\bib{Avramov/Foxby/Herzog:1994}{article}{
   author={Avramov, Luchezar L.},
   author={Foxby, Hans-Bj\o rn},
   author={Herzog, Bernd},
   title={Structure of local homomorphisms},
   journal={J. Algebra},
   volume={164},
   date={1994},
   number={1},
   pages={124--145},
   issn={0021-8693},
   review={\MR{1268330}},
   doi={10.1006/jabr.1994.1057},
}

\bib{Avramov/Iyengar/Miller:2006}{article}{
   author={Avramov, Luchezar L.},
   author={Iyengar, Srikanth},
   author={Miller, Claudia},
   title={Homology over local homomorphisms},
   journal={Amer. J. Math.},
   volume={128},
   date={2006},
   number={1},
   pages={23--90},
   issn={0002-9327},
   review={\MR{2197067}},
}

\bib{Bhatt/Hochster/Ma:2100}{article}{
author={Bhatt, Bhargav},
author={Ma, Linquan},
author={Hochster, Melvin},
title={Lim Cohen-Macaulay sequences},
status={in preparation},
}

\bib{BourbakiCAIX:2006}{book}{
   author={Bourbaki, N.},
   title={\'{E}l\'{e}ments de math\'{e}matique. Alg\`ebre commutative. Chapitres 8 et 9},
   language={French},
   note={Reprint of the 1983 original},
   publisher={Springer, Berlin},
   date={2006},
   pages={ii+200},
   isbn={978-3-540-33942-7},
   isbn={3-540-33942-6},
   review={\MR{2284892}},
}

\bib{Brennan/Herzog/Ulrich:1987}{article}{
   author={Brennan, Joseph P.},
   author={Herzog, J\"{u}rgen},
   author={Ulrich, Bernd},
   title={Maximally generated Cohen-Macaulay modules},
   journal={Math. Scand.},
   volume={61},
   date={1987},
   number={2},
   pages={181--203},
   issn={0025-5521},
   review={\MR{947472}},
   doi={10.7146/math.scand.a-12198},
}


\bib{Bruns/Herzog:1998}{book}{
   author={Bruns, Winfried},
   author={Herzog, J{\"u}rgen},
   title={Cohen-Macaulay rings},
   series={Cambridge Studies in Advanced Mathematics},
   volume={39},
   edition={2},
   publisher={Cambridge University Press, Cambridge},
   date={1998},
   pages={xii+403},
   isbn={0-521-41068-1},
   review={\MR{1251956}},
}


\bib{Cox:1995}{article}{
   author={Cox, David A.},
   title={The homogeneous coordinate ring of a toric variety},
   journal={J. Algebraic Geom.},
   volume={4},
   date={1995},
   number={1},
   pages={17--50},
   issn={1056-3911},
   review={\MR{1299003}},
}

\bib{Dao/Kurano:2014}{article}{
   author={Dao, Hailong},
   author={Kurano, Kazuhiko},
   title={Hochster's theta pairing and numerical equivalence},
   journal={J. K-Theory},
   volume={14},
   date={2014},
   number={3},
   pages={495--525},
   issn={1865-2433},
   review={\MR{3349324}},
   doi={10.1017/is014006030jkt273},
}

\bib{Fulton:1998}{book}{
   author={Fulton, William},
   title={Intersection theory},
   series={Ergebnisse der Mathematik und ihrer Grenzgebiete. 3. Folge. A
   Series of Modern Surveys in Mathematics [Results in Mathematics and
   Related Areas. 3rd Series. A Series of Modern Surveys in Mathematics]},
   volume={2},
   edition={2},
   publisher={Springer-Verlag, Berlin},
   date={1998},
   pages={xiv+470},
   isbn={3-540-62046-X},
   isbn={0-387-98549-2},
   review={\MR{1644323}},
   doi={10.1007/978-1-4612-1700-8},
}

\bib{Gillet/Soule:1987}{article}{
   author={Gillet, H.},
   author={Soul\'{e}, C.},
   title={Intersection theory using Adams operations},
   journal={Invent. Math.},
   volume={90},
   date={1987},
   number={2},
   pages={243--277},
   issn={0020-9910},
   review={\MR{910201}},
   doi={10.1007/BF01388705},
}

\bib{sga2:2005}{book}{
   author={Grothendieck, Alexander},
   title={Cohomologie locale des faisceaux coh\'{e}rents et th\'{e}or\`emes de
   Lefschetz locaux et globaux (SGA 2)},
   language={French},
   series={Documents Math\'{e}matiques (Paris) [Mathematical Documents (Paris)]},
   volume={4},
   note={S\'{e}minaire de G\'{e}om\'{e}trie Alg\'{e}brique du Bois Marie, 1962;
   Augment\'{e} d'un expos\'{e} de Mich\`ele Raynaud. [With an expos\'{e} by Mich\`ele
   Raynaud];
   With a preface and edited by Yves Laszlo;
   Revised reprint of the 1968 French original},
   publisher={Soci\'{e}t\'{e} Math\'{e}matique de France, Paris},
   date={2005},
   pages={x+208},
   isbn={2-85629-169-4},
   review={\MR{2171939}},
}

\bib{Backelin/Herzog/Ulrich:1991}{article}{
   author={Herzog, J.},
   author={Ulrich, B.},
   author={Backelin, J.},
   title={Linear maximal Cohen-Macaulay modules over strict complete
   intersections},
   journal={J. Pure Appl. Algebra},
   volume={71},
   date={1991},
   number={2-3},
   pages={187--202},
   issn={0022-4049},
   review={\MR{1117634}},
   doi={10.1016/0022-4049(91)90147-T},
}

\bib{Hochster:2017}{article}{
   author={Hochster, Melvin},
   title={Homological conjectures and lim Cohen-Macaulay sequences},
   conference={
      title={Homological and computational methods in commutative algebra},
   },
   book={
      series={Springer INdAM Ser.},
      volume={20},
      publisher={Springer, Cham},
   },
   date={2017},
   pages={173--197},
   review={\MR{3751886}},
}

\bib{Iyengar:1999}{article}{
   author={Iyengar, S.},
   title={Depth for complexes, and intersection theorems},
   journal={Math. Z.},
   volume={230},
   date={1999},
   number={3},
   pages={545--567},
   issn={0025-5874},
   review={\MR{1680036}},
   doi={10.1007/PL00004705},
}

\bib{Iyengar/Walker:2018}{article}{
   author={Iyengar, Srikanth B.},
   author={Walker, Mark E.},
   title={Examples of finite free complexes of small rank and small
   homology},
   journal={Acta Math.},
   volume={221},
   date={2018},
   number={1},
   pages={143--158},
   issn={0001-5962},
   review={\MR{3877020}},
   doi={10.4310/ACTA.2018.v221.n1.a4},
}

\bib{Kunz:2008}{book}{
   author={Kunz, Ernst},
   title={Residues and duality for projective algebraic varieties},
   series={University Lecture Series},
   volume={47},
   note={With the assistance of and contributions by David A. Cox and Alicia
   Dickenstein},
   publisher={American Mathematical Society, Providence, RI},
   date={2008},
   pages={xiv+158},
   isbn={978-0-8218-4760-2},
   review={\MR{2464546}},
   doi={10.1090/ulect/047},
}

\bib{Kurano:2004}{article}{
   author={Kurano, Kazuhiko},
   title={Numerical equivalence defined on Chow groups of Noetherian local
   rings},
   journal={Invent. Math.},
   volume={157},
   date={2004},
   number={3},
   pages={575--619},
   issn={0020-9910},
   review={\MR{2092770}},
   doi={10.1007/s00222-004-0361-8},
}


\bib{Kurano:1996}{article}{
   author={Kurano, Kazuhiko},
   title={A remark on the Riemann-Roch formula on affine schemes associated
   with Noetherian local rings},
   journal={Tohoku Math. J. (2)},
   volume={48},
   date={1996},
   number={1},
   pages={121--138},
   issn={0040-8735},
   review={\MR{1373176}},
   doi={10.2748/tmj/1178225414},
}

\bib{Kurano/Roberts:2000}{article}{
   author={Kurano, Kazuhiko},
   author={Roberts, Paul C.},
   title={Adams operations, localized Chern characters, and the positivity
   of Dutta multiplicity in characteristic $0$},
   journal={Trans. Amer. Math. Soc.},
   volume={352},
   date={2000},
   number={7},
   pages={3103--3116},
   issn={0002-9947},
   review={\MR{1707198}},
   doi={10.1090/S0002-9947-00-02589-7},
}

\bib{Lech:1960}{article}{
   author={Lech, Christer},
   title={Note on multiplicities of ideals},
   journal={Ark. Mat.},
   volume={4},
   date={1960},
   pages={63--86 (1960)},
   issn={0004-2080},
   review={\MR{140536}},
   doi={10.1007/BF02591323},
}

\bib{Ma:2020}{article}{
author={Ma, Linquan},
title={Lim Ulrich sequence: A proof of Lech's conjecture for graded base rings},
eprint={https://arxiv.org/abs/2005.02338},
}

\bib{Ma:2014}{book}{
author={Ma, Linquan},
title={The Frobenius endomorphism and multiplicities},
date={2014},
note={Thesis (Ph.\ D.)--University of Michigan},
}


\bib{Ma:2017}{article}{
   author={Ma, Linquan},
   title={Lech's conjecture in dimension three},
   journal={Adv. Math.},
   volume={322},
   date={2017},
   pages={940--970},
   issn={0001-8708},
   review={\MR{3720812}},
   doi={10.1016/j.aim.2017.10.032},
}

\bib{McDonnell:2011}{article}{
   author={McDonnell, Lori},
   title={A note on a conjecture of Watanabe and Yoshida},
   journal={Comm. Algebra},
   volume={39},
   date={2011},
   number={11},
   pages={3993--3997},
   issn={0092-7872},
   review={\MR{2855106}},
   doi={10.1080/00927872.2010.512587},
}

\bib{Peskine/Szpiro:1973}{article}{
   author={Peskine, C.},
   author={Szpiro, L.},
   title={Dimension projective finie et cohomologie locale. Applications \`a
   la d\'{e}monstration de conjectures de M. Auslander, H. Bass et A.
   Grothendieck},
   language={French},
   journal={Inst. Hautes \'{E}tudes Sci. Publ. Math.},
   number={42},
   date={1973},
   pages={47--119},
   issn={0073-8301},
   review={\MR{374130}},
}

\bib{Piepmeyer/Walker:2009}{article}{
   author={Piepmeyer, Greg},
   author={Walker, Mark E.},
   title={A new proof of the New Intersection Theorem},
   journal={J. Algebra},
   volume={322},
   date={2009},
   number={9},
   pages={3366--3372},
   issn={0021-8693},
   review={\MR{2567425}},
   doi={10.1016/j.jalgebra.2007.09.015},
}

\bib{Roberts:1998}{book}{
   author={Roberts, Paul C.},
   title={Multiplicities and Chern classes in local algebra},
   series={Cambridge Tracts in Mathematics},
   volume={133},
   publisher={Cambridge University Press, Cambridge},
   date={1998},
   pages={xii+303},
   isbn={0-521-47316-0},
   review={\MR{1686450}},
   doi={10.1017/CBO9780511529986},
}

\bib{Roberts:1987}{article}{
   author={Roberts, Paul},
   title={Le th\'{e}or\`eme d'intersection},
   language={French, with English summary},
   journal={C. R. Acad. Sci. Paris S\'{e}r. I Math.},
   volume={304},
   date={1987},
   number={7},
   pages={177--180},
   issn={0249-6291},
   review={\MR{880574}},
}

\bib{Roberts:1989}{article}{
   author={Roberts, Paul},
   title={Intersection theorems},
   conference={
      title={Commutative algebra},
      address={Berkeley, CA},
      date={1987},
   },
   book={
      series={Math. Sci. Res. Inst. Publ.},
      volume={15},
      publisher={Springer, New York},
   },
   date={1989},
   pages={417--436},
   review={\MR{1015532}},
}

\bib{Tabaa:1984}{article}{
   author={Taba\^{a}, Mohamed},
   title={Sur les homomorphismes d'intersection compl\`ete},
   language={French, with English summary},
   journal={C. R. Acad. Sci. Paris S\'{e}r. I Math.},
   volume={298},
   date={1984},
   number={18},
   pages={437--439},
   issn={0249-6291},
   review={\MR{750740}},
}


\bib{Walker:2017}{article}{
   author={Walker, Mark E.},
   title={Total Betti numbers of modules of finite projective dimension},
   journal={Ann. of Math. (2)},
   volume={186},
   date={2017},
   number={2},
   pages={641--646},
   issn={0003-486X},
   review={\MR{3702675}},
   doi={10.4007/annals.2017.186.2.6},
}

\bib{Weibel:2013}{book}{
   author={Weibel, Charles A.},
   title={The $K$-book},
   series={Graduate Studies in Mathematics},
   volume={145},
   note={An introduction to algebraic $K$-theory},
   publisher={American Mathematical Society, Providence, RI},
   date={2013},
   pages={xii+618},
   isbn={978-0-8218-9132-2},
   review={\MR{3076731}},
   doi={10.1090/gsm/145},
}


\bib{Yhee:2021}{article}{
author={Yhee, Farrah C.},
title={Ulrich modules and weakly lim Ulrich sequences do not always exist},
eprint={https://arxiv.org/abs/2104.05766},
}

\end{biblist}
\end{bibdiv}
\end{document}